\documentclass[11pt]{article}

\usepackage{tabularx} 
\usepackage{amsmath}
\usepackage{mleftright}
\usepackage{amsthm}
\usepackage{xypic}
\usepackage{latexsym}
\usepackage[plain,vlined]{algorithm2e}
\usepackage{changepage}
\usepackage{amssymb}
\usepackage{graphicx} 
\usepackage[margin=1.2in,letterpaper]{geometry} 
\usepackage{cite}
\usepackage{comment}
\usepackage{enumitem}
\usepackage{tikz}
\usetikzlibrary{arrows,automata,positioning}
\usetikzlibrary{arrows.meta}
\usepackage{subcaption}
\usepackage[colorlinks=false,pdfborder={0 0 1}]{hyperref}
\usepackage{cleveref}

\usepackage[numbers,sort&compress]{natbib}

\definecolor{applegreen}{rgb}{0.55, 0.71, 0.0}
\definecolor{lightapplegreen}{rgb}{0.95, 1, 0.9}
\definecolor{alizarin}{rgb}{0.82, 0.1, 0.26}
\definecolor{lightalizarin}{rgb}{1, 0.95, 0.98}
\definecolor{ablue}{rgb}{0.36, 0.54, 0.66}
\definecolor{lightablue}{rgb}{0.92, 0.95, 0.97}
\usepackage{empheq}

\DeclareMathOperator*{\argmin}{arg\,min}
\newtheorem{theorem}{Theorem}[section]
\newtheorem{corollary}[theorem]{Corollary}
\newtheorem{defi}[theorem]{Definition}

\newtheorem{lemma}[theorem]{Lemma}
\newtheorem{remark}[theorem]{Remark}
\newtheorem*{example}{Example}

\usepackage{xcolor}

\newcommand{\cM}{\mathcal{M}}

\newcommand{\cP}{\mathcal{P}}
\newcommand{\cE}{\mathcal{E}}

\newcommand{\cG}{\mathcal{G}}
\newcommand{\cT}{S}
\newcommand{\calH}{\mathcal{H}}
\newcommand{\R}{\mathbb{R}}
\newcommand{\N}{\mathbb{N}}
\newcommand{\vertiii}[1]{{\left\vert\kern-0.25ex\left\vert\kern-0.25ex\left\vert #1 
		\right\vert\kern-0.25ex\right\vert\kern-0.25ex\right\vert}}

\newcommand{\RR}{(-\infty,+\infty]}
\newcommand{\wsto}{\overset{w,2}{\rightharpoonup}}

\DeclareMathOperator{\prox}{prox}

\numberwithin{equation}{section}

\title{\textbf{Inexact JKO and proximal-gradient algorithms in the Wasserstein space}}
\author{
	Simone Di Marino\thanks{Dipartimento di Matematica, Università di Genova, Via Dodecaneso 35, 16146 Genova, Italy.}
	\thanks{\texttt{simone.dimarino@unige.it}} \and
	Emanuele Naldi\footnotemark[1]
	\thanks{\texttt{emanuele.naldi@edu.unige.it}} \and
	Silvia Villa\footnotemark[1]
	\thanks{\texttt{silvia.villa@unige.it}}
}
\date{}

\begin{document}

	\maketitle
	
\begin{abstract}
	This paper studies the convergence properties of the inexact Jordan-Kinderlehrer-Otto (JKO) scheme and proximal-gradient algorithm in the context of Wasserstein spaces. The JKO scheme, a widely-used method for approximating solutions to gradient flows in Wasserstein spaces, typically assumes exact solutions to iterative minimization problems. However, practical applications often require approximate solutions due to computational limitations. This work focuses on the convergence of the scheme to minimizers for the underlying functional and addresses these challenges by analyzing two types of inexactness: errors in Wasserstein distance and errors in energy functional evaluations. The paper provides rigorous convergence guarantees under controlled error conditions, demonstrating that weak convergence can still be achieved with inexact steps. The analysis is further extended to proximal-gradient algorithms, showing that convergence is preserved under inexact evaluations.
\end{abstract}	


\section{Introduction}

The Proximal Point Algorithm (PPA) in the $2$-Wasserstein space, also known as Jordan-Kinderlehrer-Otto (JKO) scheme \cite{JKO98} or Minimizing Movement scheme \cite{AGS08,DeGiorgi93}, is a well-known variational method for approximating solutions to gradient flows in the Wasserstein space, and as such it is often used in the study of partial differential equations (PDEs) via optimal transport. Given a functional $\cG$ defined over the space of probability measures, the JKO scheme approximates the gradient flow of $\cG$ in the Wassertein space by iteratively solving a sequence of minimization problems of the form
\begin{equation}\label{eq:prox_point_intro}
	\mu_{n+1} = J_{\tau_n}(\mu_n) := \argmin_{\nu \in \mathcal{P}_2(\mathbb{R}^d)} \left\{ \cG(\nu) + \frac{1}{2\tau_n} W_2^2(\nu, \mu_n) \right\},
\end{equation}
where $ W_2 $ denotes the 2-Wasserstein distance, $\{ \tau_n \}_n$ is a positive sequence of stepsize parameters, and $ \mathcal{P}_2(\mathbb{R}^d) $ is the space of probability measures with finite second moments.\\
While the first use of the JKO scheme was to prove existence results for nonlinear PDEs and their convergence properties, the importance of studying JKO schemes extends beyond purely theoretical considerations. In recent years, gradient flows in the space of probability measures have found applications in machine learning \cite{RM17,ZCL18}, neural networks
optimization \cite{ChizatBach18,MMM19,LLOM21}, and sampling \cite{Bernton18,CB18,DMM18,Wibisono18}. As an example, Langevin dynamics describe the evolution of a probability distribution according to a stochastic differential equation that models the behavior of particles in a potential field with added Brownian motion. The Langevin equation can be interpreted as a gradient flow of the Kullback-Leibler (KL) divergence with respect to the Wasserstein-2 metric \cite{Wibisono18}. The JKO scheme and the proximal-gradient algorithm in Wasserstein spaces \cite{WPG2020} provide natural ways to discretize Langevin dynamics by iteratively solving proximal minimization problems related to the KL divergence.
Although discretized Wasserstein gradient flows have been proposed in the literature \cite{AGS08,Bowles15,CL17,JKO98,MRV11,Wibisono18}, most of them have not been studied as minimization algorithms. In \cite{NS2021} weak convergence results for JKO were established. In \cite{WPG2020} estimates for the convergence of a proximal-gradient algorithm were established and later extended to weak convergence in the convex case in \cite{diao2023} only for the Bures–Wasserstein space, that is the (closed and geodesically convex) subset of Gaussians in the Wasserstein space.

Unfortunately, in Wasserstein spaces there are very few cases where the minimization problem \eqref{eq:prox_point_intro} can be computed in closed form. A notable example is the case where $\mathcal{\cG}(\mu) = \int g \, \mathrm{d}\mu$ and $g$ is proper, convex, lower semicontinuous and its proximity operator can be computed in a closed form\footnote{See for example the repository at www.proximity-operator.net.}, see \cite{Bowles15}. Other cases are when the input and the functional are specific, for example when $\mathcal{G}$ is the negative entropy and the input $\mu$ is Gaussian, see \cite[Example 8]{Wibisono18}.\\
For this reason, in the Wasserstein setting it is necessary to devise algorithms to compute inexact JKOs, where the exact solution to the minimization problem is replaced by an approximate one. A first option is to regularize problem \eqref{eq:prox_point_intro}, making it possibly easier to solve. One way is to substitute the Wasserstein distance $W_2$ in \eqref{eq:prox_point_intro} with its entropic regularized counterpart $W_2^\epsilon$ (see \cite{Leonard14,KS67}). The resulting JKO scheme is analyzed in \cite{CDPS17}. However, the main focus of the existing literature is about constructing a sequence close to the Wasserstein Gradient Flow. Other recent strategies rely on regularizing the associated Hamilton-Jacobi equation, motivated by regularized proximal operators in Euclidean space \cite{OHW23,HWO24}, and applying Hopf-Cole type transformations to obtain an effective solution strategy, see \cite{LLO23}. Another regularization strategy can be found in \cite{LLW20}, while modern ways to approximate \eqref{eq:prox_point_intro}, involving neural networks, are introduced in \cite{LWL24}. In practice, approximate solutions are unavoidable, since numerical solvers typically achieve a solution only up to a specified tolerance. Optimization schemes that tackle directly the original problem usually make use of the Benamou-Brenier formula \cite{BB2000} to rewrite the $2$-Wasserstein distance and solve a saddle point problem. A well-known method to approximate a solution of \eqref{eq:prox_point_intro} is the so-called ``ALG2" introduced in \cite{BCL16}, which makes use of an ADMM-type algorithm, but other solvers can be applied to the saddle point problem \cite{Chambolle2011}. In all these cases, the solution to \eqref{eq:prox_point_intro} is computed only approximately. Despite their practical relevance, the existing convergence analyses for JKO-based algorithms predominantly assume that $\mu_{n+1}=J_{\tau_n}(\mu_n)$, that is, each JKO step is computed exactly. However, as previously discussed, in all existing JKO-based algorithms it holds instead \[\mu_{n+1}\approx J_{\tau_n}(\mu_n).\]
For this reason, understanding the behavior of inexact JKO schemes and quantifying their convergence properties is a crucial step towards broader applicability. 

\subsection{Inexact JKO}

In this work, we address this gap by analyzing two types of errors in the computation of the JKO scheme. The first type involves an error with respect to the Wasserstein distance, where the approximate solution $ \mu_{n+1} $ satisfies
\begin{equation}\label{eq:err1}
	W_2(\mu_{n+1}, J_{\tau_n}(\mu_n)) \leq \epsilon_n,
\end{equation}
with $ J_{\tau_n}(\mu_n) $ being the exact JKO step defined in \eqref{eq:prox_point_intro} and $\{\epsilon_n\}_n$ a nonnegative sequence. We refer to this discrepancy as a \emph{distance-type error}. Our first main result concerns the convergence behavior under this type of approximation.
\begin{theorem}[Convergence for distance-type error]\label{thm:JKO_intro_convergence1}
	Let $\cG : \cP_2(\R^d) \to \mathbb{R} \ $ be proper, lower semicontinuous and convex along generalized geodesics with $\argmin \cG \neq \emptyset$. Let $\{\epsilon_n\}_n\subset \mathbb{R}_{\geq 0}$ with $\sum_{n=0}^{\infty} \epsilon_n < \infty$ and let $\{\tau_n\}_n \subset \mathbb{R}_{> 0}$ with $\sum_{i=0}^{\infty}\tau_i= \infty$. Define $\sigma_n:= \sum_{i=0}^{n-1}\tau_i$, for $n\in\mathbb{N}$ and suppose $\sum_{n=1}^{\infty}\frac{\sigma_n}{\tau_n}\epsilon_{n-1}^2  < \infty$. Let $\{\mu_n\}_n$ satisfying \eqref{eq:err1}, then 
	\[\cG(J_{\tau_n}(\mu_n))-\inf \cG = O\left(\frac{1}{\sigma_n}\right), \quad \text{for } n\to \infty,\]
	and $W_p(\mu_n,\mu^*)\to 0$ for all $p\in [1,2)$, where $\mu^*\in \argmin \cG$.
\end{theorem}

This establishes convergence for a first inexact version of the JKO algorithm, under assumptions on the stepsize and error sequences that are analogous to those used in the analysis within Hilbert spaces. For the sequence $\{J_{\tau_n}(\mu_n)\}_n$, we are able to derive convergence rates in terms of the objective functional $\cG$ in relation to the sequence of partial sums of stepsizes $\{\sigma_n\}_n$, mirroring what is known for Hilbert spaces. While this first choice of error seems natural, a second slightly more restrictive choice can be more expressive. In particular, in the next theorem we provide convergence rates for the sequence $\{\mu_n\}_n$ which is the real sequence we have actually access to. Moreover, algorithms that try to solve the minimization problem in \eqref{eq:prox_point_intro}, have as objective to minimize the energy $\cG(\cdot) + \frac{1}{2\tau_n}W_2^2(\cdot, \mu_n)$ and can sometimes provide convergence properties in terms of this fuctional. For this reason, the second type of error we consider, is thus measured in terms of the energy of the minimization problem in \eqref{eq:prox_point_intro} and we refer to it as the \emph{variational-type error}
\begin{equation}\label{eq:err2}
	\cG(\mu_{n+1}) + \frac{1}{2\tau_n} W_2^2(\mu_{n+1}, \mu_n) \leq \cG(J_{\tau_n}(\mu_n)) + \frac{1}{2\tau_n} W_2^2(J_{\tau_n}(\mu_n), \mu_n) + \epsilon_n^2.
\end{equation}
For this error, we will obtain the following result.
\begin{theorem}[Convergence for variational-type error]\label{thm:JKO_intro_convergence2}
	Let $\cG : \cP_2(\R^d) \to \mathbb{R} \ $ proper, lower semicontinuous and convex along generalized geodesics with $\argmin \cG \neq \emptyset$. Let $\{\epsilon_n\}_n\subset \mathbb{R}_{\geq 0}$ with $\sum_{n=0}^{\infty} \epsilon_n < \infty$, $\{\tau_n\}_n \subset \mathbb{R}_{> 0}$ with $\sum_{i=0}^{\infty}\tau_i= \infty$ and let $\sigma_n:= \sum_{i=0}^{n-1}\tau_i$, for $n\in\mathbb{N}$ and $\sum_{n=1}^{\infty}\frac{\sigma_n}{\tau_n}\epsilon_{n}^2  < \infty$. Let $\{\mu_n\}_n$ satisfying \eqref{eq:err2}, then
	\[\cG(\mu_n)-\inf \cG = O\left(\frac{1}{\sigma_n}\right), \quad \text{for } n\to \infty,\]
	and $W_p(\mu_n,\mu^*)\to 0$ for all $p\in [1,2)$, where $\mu^*\in \argmin \cG$.
\end{theorem}
In contrast to the statement  
of \Cref{thm:JKO_intro_convergence1}, \Cref{thm:JKO_intro_convergence2} shows that variational-type errors allow us to establish convergence for the sequence of the values $\cG(\mu_n)$ directly, rather than for $\cG(J_{\tau_n}(\mu_n))$. This distinction is practically significant, as $J_{\tau_n}(\mu_n)$  is not available in actual computations, whereas $\mu_n$
is the output actually produced by the algorithm.
\subsection{Inexact proximal-gradient}
In the last part of the paper, we focus on the composite problem
	\begin{equation}
		\label{eq:psum_intro}
		\min_{\mu\in\cP_2(\R^d)} \cG (\mu) = \mathcal{E}_F(\mu)+\mathcal{H}(\mu)
	\end{equation}
	with $\mathcal{E}_F(\mu)= \int F \, \mathrm{d}\mu$, $F:\R^d\to \R$. The prototypical example of the functional $\cG$ is the free energy functional (a key example also in \cite{WPG2020}) and it is intimately related to the Kullback-Leibler divergence.   The free energy is expressed as
	\[\mu \mapsto \int F(x) \, d\mu(x) + \text{Ent}(\mu),\]
	where $\text{Ent}(\mu)=\int \log(\mu(x)) d\mu(x)$ if $\mu$ is absolutely continuous with respect to the Lebesgue measure and has density $\mu$, and $\text{Ent}(\mu)=+\infty$ otherwise. The entropy functional is proper, lower semicontinuous and convex along generalized geodesics and by definition, its domain satisfies $\text{dom}(\text{Ent})\subset \cP_2^r(\mathcal{X})$, hypothesis that we will assume on $\mathcal{H}$. As highlighted in \cite{WPG2020}, this example is related to the Langevin dynamic and the Fokker-Planck equation. 
	
	The proximal gradient algorithm, originally devised in Hilbert spaces to minimize  sums of a smooth and a nonsmooth function similarly to \eqref{eq:psum_intro} has been extended to address problems on the Wasserstein space. This extension was first introduced  in \cite{Wibisono18} and further explored in \cite{WPG2020}.
	The method consists in two alternating steps: a gradient-descent (forward) step for the functional $\cE_F$ and a proximal (backward) step for $\mathcal{H}$, and can be written as $\mu_{n+1}=J_{\tau , \calH}\left( (I-\tau \nabla F)_{\#}(\mu_n)\right)$. Clearly this scheme generalizes the proximal point method (JKO scheme), and the additional assumption we take on the domain of $\mathcal{H}$ is the sole reason for separating the  analysis of these two algorithms.
	However, as we already observed for the JKO scheme, in practice it is generally not feasible to implement an exact proximal-gradient algorithm for this functional (see \cite[Section 4.1]{Wibisono18}) . This limitation then strongly motivates again the introduction of an inexact proximal-gradient scheme.
	
	We consider inexact schemes that perform iterations of the type
	\[\mu_{n+1} \approx J_{\tau_n , \calH}\left((I-\tau_n \nabla F)_{\#}(\mu_n)\right),\]
	with a positive sequence of stepsizes $\{\tau_n\}_{n}$. In this case, the introduction of a variable stepsize can have practical relevance. In fact, Wibisono notes in \cite[Section 2.2.2]{Wibisono18} that the classical unadjusted Langevin algorithm (ULA) can be interpreted as performing a gradient step for the potential energy and a ``flow" step for the entropy functional. Since for small stepsizes the JKO step closely approximates the flow step \cite{AGS08}, the whole ULA iteration can actually be interpreted as an inexact step of a proximal-gradient algorithm in Wasserstein spaces. However, it is well-known that the ULA procedure introduces a bias and the algorithm converges to an incorrect distribution. This drawback motivates the analysis of variable (vanishing) stepsizes, as they are sometimes  used in practice to ``adjust” the ULA scheme and drive convergence towards the correct distribution.
	
	For the convergence analysis of the scheme, we build upon the results of \cite{WPG2020}, extending them  to the convex (and not necessarily strongly convex) setting and to the inexact setting. We establish analogues of \Cref{thm:JKO_intro_convergence1} and \Cref{thm:JKO_intro_convergence2} also for the inexact proximal-gradient algorithm, providing convergence guarantees for the resulting sequence $\{\mu_n\}_n$ along with corresponding convergence rates.\\
	
	For both the algorithms we consider in this work, the results we provide have a long history in Hilbert spaces. In the original works of Martinet \cite{Martinet72} and Rockafellar \cite{Rockafellar76}, convergence for proximal point algorithms with summable errors were introduced. The impact of errors on convergence has been further analyzed in \cite{Aus87,Gul91,Com97,SolSva00,SolSva01,salzo2012inexact}. The extension to Banach spaces and Bregman divergences has also been considered in the works \cite{AlbBurIus97,Eck98,BurSva01} while an analysis of inexact evaluations for proximal-gradient algorithms can be found in \cite{CombettesWajs2005,Villa2013,Devolder2014}.  The analysis in Wasserstein spaces, however, is more subtle, as we will discuss throughout the paper. In particular, in Wasserstein spaces it is not possible to rely on the nonexpansivity propriety of the operator $J_{\tau_n}$ and thus it is not possible to use directly the correspondent of classical analysis in Hilbert spaces \cite{Combettes2004}.

\subsection{Contributions and structure of the paper} 
The main contributions of this paper can be summarized as follows:
\begin{itemize}
	\item We propose an inexact JKO framework, where the minimization problems solved at each step allow for controlled approximations in either the Wasserstein distance or energy functional evaluation. We rigorously analyze the convergence of the resulting schemes and provide sufficient conditions for weak convergence. We also discuss rates on the objective functional $\cG$.
	\item We extend the analysis to proximal-gradient algorithms in Wasserstein spaces. Based on the work in \cite{WPG2020} we first provide a finer discrete EVI for the proximal-gradient algorithm. With this, we demostrate how the inexact evaluation of proximal steps can still guarantee convergence under suitable assumptions on the error sequence. The results we obtain expand the ones in \cite{WPG2020} and \cite[Theorem 5.3]{diao2023}.
	\item Both the inexact JKO and proximal-gradient are analyzed with varying stepsizes, which result in a new and interesting analysis, parallel, but with some key differences, to the classical one in Hilbert spaces.
\end{itemize}
Even if the proximal-gradient algorithm in Wasserstein spaces is more general than the JKO scheme, we decided to keep the contribution separated. The reason lies in the fact that in the analysis of the proximal-gradient algorithm we assume an additional regularity assumption (the same present also in \cite{WPG2020}), while we do not need such assumption for the analysis of the JKO algorithm.\\

The remainder of this paper is organized as follows. In \Cref{sec:preliminaries}, we introduce the theoretical background on optimal transport, gradient flows in Wasserstein spaces, and the classical JKO scheme, while also introducing the weak topology we consider in this paper. In this section, we also provide in \Cref{thm:introduction} generalizations of known results to fit our setting. \Cref{sec:JKO} introduces the inexact JKO framework, detailing the types of errors considered and presenting the main convergence results. In \Cref{sec:PG}, we extend our approach to proximal-gradient algorithms in Wasserstein spaces and establish convergence guarantees for inexact proximal-gradient iterations.

\section{Preliminaries}\label{sec:preliminaries}

We denote by $\mathcal{M}(\R^d)$ the Banach space of bounded measures defined on $\mathcal{B}(\mathbb{R}^d)$, the Borel $\sigma$-algebra of $\R^d$. We define $\mathcal{M}_2(\R^d)$ as the subspace of measures with finite second moments, i.e.,
\[\mathcal{M}_2(\R^d) = \left\{\mu \in \mathcal{M}(\R^d) \ \Big| \ \int_{\R^d} \|x\|^2 \,\mathrm{d} |\mu|(x) < +\infty \right\},\]
where $|\mu|$ denotes the total variation of $\mu$. We write $\mathcal{P}(\R^d)$ for the subset of $\mathcal{M}(\R^d)$ of probability measures (i.e., positive measures with mass $1$) and we define $\mathcal{P}_2(\R^d) := \mathcal{P}(\R^d) \cap \mathcal{M}_2(\R^d)$.

For every $\mu \in \mathcal{P}_2(\mathbb{R}^d)$, $L^2(\mu)$ denotes the space of functions $f : (\mathbb{R}^d, \mathcal{B}(\mathbb{R}^d)) \to (\mathbb{R}^d, \mathcal{B}(\mathbb{R}^d))$ such that $\int \|f\|^2 \, \mathrm{d}\mu < + \infty$. For every $\mu \in \mathcal{P}_2(\mathbb{R}^d)$, we denote by $\| \cdot \|_{\mu}$ and $\langle \cdot, \cdot \rangle_{\mu}$ respectively the norm and the inner product of the space $L^2(\mu)$. For any measures $\mu, \nu$, we write $\mu \ll \nu$ if $\mu$ is absolutely continuous with respect to $\nu$, and we denote $\mathcal{L}^d$ the Lebesgue measure over $\mathbb{R}^d$. The set of absolutely continuous measures with respect to Lebesgue, within $\mathcal{P}_2(\R^d)$, is denoted by $\mathcal{P}_2^r(\mathbb{R}^d) := \{\mu \in \mathcal{P}_2(\mathbb{R}^d), \mu \ll \mathcal{L}^d\}$. Throughout the paper we will also refer to such measures as regular measures.

For every measurable map $T:(\mathbb{R}^d, \mathcal{B}(\mathbb{R}^d))\to (\mathbb{R}^m, \mathcal{B}(\mathbb{R}^m))$ and for every $\mu \in \mathcal{P}_2(\mathbb{R}^d)$, we denote by $T_{\#}\mu\in \mathcal{P}_2(\R^m)$ the pushforward measure of $\mu$ by $ T $ characterized by the 'transfer lemma', i.e., $
\int_{\R^m} \varphi(y) dT_{\#} \mu(y) = \int_{\R^d} \varphi(T(x)) d\mu(x)$, for any measurable and bounded function $\varphi$.

Let $p\in [1,2]$ and consider the p-Wasserstein distance defined for every $\mu, \nu \in \mathcal{P}_2(\mathbb{R}^d)$ by
\begin{equation}\label{eq:W_distiance}
	W_p^p(\mu, \nu) := \inf_{\gamma \in \Gamma(\mu, \nu)} \int_{\mathbb{R}^d \times \mathbb{R}^d} \|x - y\|^p \, \mathrm{d}\gamma(x, y),
\end{equation}
where $\Gamma(\mu, \nu)$ is the set of couplings between $\mu$ and $\nu$ \cite{Villani08}, i.e., the set of nonnegative measures $\gamma$ over $\mathbb{R}^d \times \mathbb{R}^d$ such that $ \pi^1_{\#} \gamma = \mu$ (respectively $ \pi^2_{\#} \gamma = \nu$) where $ \pi^1: (x, y) \mapsto x$ (respectively $ \pi^2: (x, y) \mapsto y$) is the projection onto the first (respectively second) component. The set $\Gamma(\mu,\nu)$ is called the set of transport plans between $\mu$ and $\nu$ and the set of plans that minimize \eqref{eq:W_distiance} is the set of optimal transport plans and denoted by $\Gamma_{\text{opt}}(\mu,\nu)$. By Jensen inequality, it is clear that $W_p(\mu,\nu)\leq W_q(\mu,\nu)$ for every $p\leq q$ and $\mu,\nu \in \cP_q(\R^d)$.

We recall a foundamental theorem, due to Brenier \cite{Brenier91} and Knott-Smith \cite{KnottSmith1984}, see also \cite[Proposition 5.2]{ABS2024}.

\begin{theorem}[Brenier, Knott-Smith]\label{thm:brenier}
	Let $p=2$, $\mu \in \mathcal{P}_2^r(\mathbb{R}^d)$ and $\nu \in \mathcal{P}_2(\mathbb{R}^d)$. Then
	\begin{itemize}
		\item[(i)] Problem \eqref{eq:W_distiance} has a unique solution $\gamma$. In addition, $\gamma$ is induced by a transport map, that is, there exists a uniquely determined $\mu$-almost everywhere map $T_{\mu}^{\nu}: \mathbb{R}^d \rightarrow \mathbb{R}^d$ such that $\gamma = (I, T_{\mu}^{\nu})_{\#} \mu$ where $(I, T_{\mu}^{\nu}): x \mapsto (x, T_{\mu}^{\nu}(x))$. Moreover $T_{\mu}^{\nu} = \nabla \psi$, where $\psi: \mathbb{R}^n \to (-\infty, +\infty]$ is a lower semicontinuous convex function differentiable $\mu$-almost everywhere. The map $T_{\mu}^{\nu}$ is called the optimal transport map from $\mu$ to $\nu$. 
		
		\item[(ii)] Conversely, if $\psi$ is convex, lower semicontinuous, and differentiable $\mu$-almost everywhere with $|\nabla \psi| \in L^2(\mu)$ and $(\nabla \psi)_{\#} \mu=\nu$, then $T_{\mu}^\nu := \nabla \psi$ is the optimal transport map from $\mu$ to $\nu$.
		
		\item[(iii)] If also $\nu \in \mathcal{P}_2^r(\mathbb{R}^d)$, then $ T_{\nu}^{\mu} \circ T_{\mu}^{\nu} = I$ $\mu$-almost everywhere and $ T_{\mu}^{\nu} \circ T_{\nu}^{\mu} = I$ $\nu$-almost everywhere.
	\end{itemize}
\end{theorem} 	

\begin{corollary}
	In the hypothesis of the previous theorem it holds
	\[
	W_2^2(\mu, \nu) = \inf_{T: T_{\#}\mu = \nu} \int_{\mathbb{R}^d} \|x - T(x)\|^2 d\mu(x).
	\]
\end{corollary}

In this paper, as it is commonly the case in the literature, we may refer to the space of probability distributions $\mathcal{P}_2(\mathbb{R}^d)$ equipped with the 2-Wasserstein distance as the Wasserstein space. In the following we define some weak topologies that can be considered on the space $\mathcal{P}_2(\mathbb{R}^d)$.

\begin{defi}[Narrow topology] The narrow topology on $\cP_2(\R^d)$ is the weak$^*$ topology of $(C_b(\R^d))'$, where
	\[C_b(\R^d):=\left\{f:\R^d\to \R \ \big| \ f \text{ is continuous and bounded}\right\},\]
	endowed with the norm
	$\|f\|_{C_b(\R^d)}:=\sup_{x \in \R^d} |f(x)|$.
\end{defi}

This topology is weaker than the strong topology induced by $W_p$, $p\in[1,2]$. In particular, we recall that whenever a sequence converges with respect to the distance $W_p$ it converges also narrowly, see \cite[Lemma 5.1.7]{AGS08}. We introduce next another topology which will allow us to state more general results throughout the paper. See for example \cite[Section 3]{NS2021} and \cite[Section 5]{BCIKNV24} for further comments.

\begin{defi}
	Let $C_2^w(\R^d)$ be the space defined by
	\[C_2^w(\R^d):=\left\{f:\R^d\to \R \ \big| \ f \text{ is continuous and } \lim_{\|x\|\to \infty}\frac{f(x)}{1+\|x\|^2}=0\right\},\]
	endowed with the norm
	$\|f\|_{C_2^w(\R^d)}:=\sup_{x \in \R^d} \frac{|f(x)|}{1+\|x\|^2}$.	
\end{defi}

It is known that $\mathcal{M}_2(\R^d)$ can be seen as the dual of such space when endowed with the norm $\|\mu\|_{\cM_2(\R^d)}=\int_{\R^d} \|x\|^2 \,\mathrm{d}|\mu|(x)$, see for example \cite[Section 5]{BCIKNV24} for a proof and further comments. In this work we consider the weak-$*$ topology in this space restricted to the subset $\mathcal{P}_2(\R^d)$ and we denote it by $\tau_{w,2}$. Whenever a sequence $\{\mu_n\}_n \subset \mathcal{P}_2(\R^d)$ is converging to $\mu \in \mathcal{P}_2(\R^d)$ with respect to this topology, we denote it by $\mu_n \wsto \mu$. Clearly, this topology is finer than the narrow topology. This means that if a sequence $\{\mu_n\}_n$ is converging to $\mu$ in $\cP_2(\R^d)$ with respect to the topology $\tau_{w,2}$, then $\mu_n\to \mu$ narrowly. Moreover, the convergence we obtain is ``weak" by name, but it implies convergence in $p$-Wasserstein distance for any $p\in [1,2)$, see \cite[Remark 7.1.11]{AGS08}.

\begin{lemma}\label{lem:topology}
	Let $\{\mu_n\}_n$ and $\{\bar \mu_n\}_n$ be two bounded sequences in $(\mathcal{P}_2(\R^d),W_2)$ such that $W_2(\bar \mu_n,\mu_n)\to 0$ and $\bar \mu_n \wsto \mu^*$, then also $\mu_n\wsto \mu^*$.
\end{lemma}
\begin{proof}
	Since both $\{\bar \mu_n\}_n$ and $\{\mu_n\}_n$ are bounded sequences in $(\mathcal{P}_2(\R^d),W_2)$, there exists $c>0$ such that \[\{\bar \mu_n\}_n\cup\{\mu_n\}_n \subset K_c :=\left\{\mu \in \cP_2(\R^d) \ \big| \ \int_{\R^d} \|x\|^2 \, \mathrm{d} \mu(x) \leq c\right\}.\] 
	By \cite[Corollary 3.6 (c)]{NS2021} the set $K_c$ is relatively sequentially compact in $\cP_2(\R^d)$ endowed with the topology $\tau_{w,2}$. From \cite[Lemma 5.1.7]{AGS08} we also have that it is sequentially closed, so that $(K_c, \tau_{w,2})$ is sequentially compact. By hypothesis we have $W_2(\bar \mu_n,\mu_n)\to 0$, which implies $W_1(\bar \mu_n,\mu_n)\to 0$. Since $\bar \mu_n \wsto \mu^*$ then $W_1(\bar \mu_n,\mu^*)\to 0$ and from the fact that $W_1$ is induced by a norm, we obtain $W_1(\mu_n,\mu^*)\to 0$. Now, from every subsequence of $\{\mu_n\}_n$ we can extract a further subsequence converging with respect to the topology $\tau_{w,2}$ to some $\mu^{**} \in K_c$, and thus also converging with respect to $W_1$. However, since $(K_c,W_1)$ is Hausdorff and $W_1(\mu_n,\mu^*)\to 0$, we can conclude that $\mu^{**}=\mu^*$, so that $\mu_n \wsto \mu^*$.
\end{proof}

\begin{theorem}[Opial property,\text{\cite[Theorem 5.1]{NS2021}}]\label{thm:Opial}
	Let $\{\mu_n\}_n$  such that $\mu_n \wsto \mu$ in $\cP_2(\R^d)$. Then
	\[W_2^2(\nu,\mu)+\liminf_{k\to \infty} W_2^2(\mu_n,\mu)\leq \liminf_{k\to \infty} W_2^2(\mu_n,\nu) \quad \text{for every } \nu \in \cP_2(\R^d).\]
\end{theorem}

From \cite[Corollary 5.3]{NS2021} we have that such property holds under the weaker condition $\mu_n \to \mu$ narrowly in $\cP_2(\R^d)$.

\begin{defi}[Geodesic]
	A (minimal, constant speed) geodesic $(\mu_t)_{t\in [0,1]}$ in $\cP_2(\R^d)$ connecting two given measures $\mu_0,\mu_1\in \cP_2(\R^d)$ is a Lipschitz curve satisfying
	\begin{equation}
		W_2(\mu_s,\mu_t)= |t-s|W_2(\mu_0,\mu_1) \quad \text{for every } s,t \in [0,1].
	\end{equation}
\end{defi}

\begin{defi}[Generalized geodesic]
	A generalized geodesic between $\mu_0,\mu_1\in \cP_2(\R^d)$ (with base $\nu\in \cP_2(\R^d)$) is a curve $(\mu_t)_{t\in [0,1]}$ in $\cP_2(\R^d)$  defined by
	\[\mu_t = (\pi_t^{2\to 3})_{\#} \gamma \quad t\in [0,1],\]
	where $\pi_t^{2\to 3}:= (1-t)\pi^2+t\pi^3$, $\gamma \in \Gamma(\nu,\mu_0,\mu_1)$, $\pi_{\#}^{1,2}\gamma \in \Gamma_{\text{opt}}(\nu,\mu_0)$ and $\pi_{\#}^{1,3}\gamma \in \Gamma_{\text{opt}}(\nu,\mu_1)$, with $\pi^{1,2}:(x,y,z)\to (x,y)$ and $\pi^{1,3}:(x,y,z)\to (x,z)$.
\end{defi}

\begin{defi}[Convexity]
	Let $\cG: \cP_2(\R^d) \to (-\infty,+\infty]$ be a proper function. $\cG$ is convex along (generalized) geodesics if for every $\mu_0,\mu_1, \nu \in \cP_2(\R^d)$ there exists a  (generalized) geodesic $(\mu_s)_{s\in[0,1]}$ between $\mu_0$ and $\mu_1$ (with base $\nu$), along which
	\begin{equation}
		\cG(\mu_s)\leq (1-s)\cG(\mu_0)+s\cG(\mu_1) \quad \text{for every } s \in [0,1].
	\end{equation}
\end{defi}

Clearly, convexity along generalized geodesics implies convexity along geodesics. Notice that in \cite[Definition 9.2.4]{AGS08}, the authors consider as base points only $\nu\in D(\cG)$, this is because they consider as inputs of $J_{\tau}$ only probabilities in $\overline{D(\cG)}$. For our purposes, however, since the input of $J_{\tau}$ will be an approximation of some element in $D(\cG)$ and can be strictly outside of the domain, we have to assume this slightly more restrictive definition. The definition we provide is actually common in the literature where some times convexity along generalized geodesics is required for base points $\nu\in \cP_2^r(\R^d)$ (see for example \cite[Section 2.3.2]{WPG2020}). The notion we require coincides with this last one whenever $\cG$ is supposed lower semicontinuous, which we will suppose throughout the paper.

\begin{example}
	A classical example of a functional that is convex along generalized geodesics is the potential energy \[\mu \mapsto \int F(x) \, d\mu(x),\]
	where $F:\R^d \to \RR$ is proper convex and lower semicontinuous. Another classical example is the interaction energy \[\mu \to \iint W(x,y) d\mu(x) \, d\mu(y),\]
	where $W: \R^d\times\R^d\to \RR$ is proper convex and lower semicontinuous. A typical choice is $W(x,y)=V(x-y)$, with $V:\R^d\times\R^d\to \RR$ proper convex and lower semicontinuous. A further relevant example we will refer again later is the entropy functional
	\[\mu \mapsto \text{Ent}(\mu):=\begin{cases}\int \log(\rho(x))\, d\rho(x) & \text{if } \mu\in \cP_2^r(\R^d) \text{ and } \mu = \rho \mathcal{L}^d,\\
		+\infty & \text{otherwise.}\end{cases}\]
	See \cite[Section 9.3 and 9.4]{AGS08} for further examples and comments about geodesically convex functionals. 
\end{example}

\begin{remark}
	The functional $W_2(\mu^1,\cdot)$ is not convex along geodesics (and therefore it is also not convex along generalized geodesics). However, as noted in \cite[Remark 9.2.8]{AGS08}, $W_2^2(\mu^1,\cdot)$ is convex along all generalized geodesics with base $\mu^1$. This property is actually essential for the well-posedness of the JKO scheme.
\end{remark}

\begin{theorem}[$\text{\cite[Theorem 4.2]{NS2021}}$]\label{thm:lsc_wlsc}
	Every lower semicontinuous and geodesically convex functional $
	\varphi : \mathcal{P}_2(\R^d) \to \RR$
	is sequentially lower semicontinuous w.r.t. the topology $\tau_{w,2}$ on $\mathcal{P}_2(\R^d)$.
\end{theorem}

\begin{defi}[Subdifferential,$\text{\cite[Section 10.1.1]{AGS08}}$] Let $\mathcal{H}$ be a functional defined on $\cP_2^r(\R^d)$ and let $\mu \in \cP_2(\R^d)$. A function $\xi \in L^2(\mu)$ belongs to the (Fréchet) subdifferential of $\mathcal{H}$ at $\mu$ iff
	\[\mathcal{H}(\nu)-\mathcal{H}(\mu)\geq \int \langle \xi (x), T_{\mu}^\nu(x)- x \rangle \, d\mu(x).\]
	When is not generating confusion, we refer to a specific element of the subdifferential of $\mathcal{H}$ at $\mu$ as $\nabla_W \mathcal{H}(\mu)$.
\end{defi}

Let $\tau > 0$, we define the (in general multivalued) operator $J_{\tau}: \cP_2(\R^d) \to \cP_2(\R^d)$, by
\begin{equation}\label{eq:prox}
	J_{\tau}(\mu) = \argmin_{\nu \in \cP_2(\R^d)} \left\{\cG(\nu) + \frac{1}{2\tau} W_2^2(\nu,\mu)\right\}.
\end{equation}

Throughout this paper, we will use a specific property (see equation \eqref{eq:discrete_EVI} below) of the operator $J_{\tau}$ that enables to find what is called a \emph{discrete EVI} for the JKO sequence. Usually, property \eqref{eq:discrete_EVI} is stated for any initial point $\mu \in \overline{D(\cG)}$ (see for example \cite[Lemma 9.2.7]{AGS08} and \cite[Theorem 4.1.2 (i) and (ii)]{AGS08}). However, in our analysis, the input $\mu$ will be only an approximation of a previous JKO iterate and, as such, may not lie within the domain of $\cG$ (or its closure). For this reason, we need to establish a more general result, for which we provide a proof.

\begin{theorem}\label{thm:introduction}
	Let $\cG : \cP_2(\R^d) \to \RR$ proper, lower semicontinuous and convex along generalized geodesics with $\argmin \cG \neq \emptyset$ and let $\tau > 0$. Then \begin{itemize}
		\item[(i)] For all $\mu \in \cP_2(\R^d)$ the minimization problem in \eqref{eq:prox} has a unique solution
		\item[(ii)] For each $\mu, \nu \in \cP_2(\R^d)$, it holds 
		\begin{equation}\label{eq:discrete_EVI}
			W_2^2(J_{\tau}(\mu),\nu) - W_2^2(\mu,\nu) \leq 2\tau \left(\cG(\nu) -  \cG(J_{\tau}(\mu))\right) - W_2^2(J_{\tau}(\mu),\mu) 
		\end{equation}
		\item[(iii)] If $\mu_n$ is converging to $\mu$ in $(\cP_2(\R^d),W_2)$, then $W_2(J_\tau(\mu_n),J_\tau(\mu)))\to 0$.
	\end{itemize}
\end{theorem}

\begin{proof} For the first two points we can refer to \cite[Lemma 9.2.7]{AGS08} and \cite[Theorem 4.1.2 (i) and (ii)]{AGS08}. As we mentioned above, in \cite{AGS08} they consider only $\mu \in \overline{D(\cG)}$. For our purposes we need the properties to hold for every $\mu $ so we provide a brief proof here.
	
	\begin{itemize}
		\item[(i)] Let us consider a minimizing sequence $\nu_n$ for \eqref{eq:prox}. Define $\cG_{\tau}(\nu)= \cG(\nu) + \frac{1}{2\tau} W_2^2(\mu,\nu) $ and $\cG_\tau^*= \inf \cG_{\tau}$ (note that $\cG_\tau^* > -\infty$ since $\cG_\tau^* \geq \inf \cG$). By definition we have $\cG_{\tau}(\nu_n) = \cG_\tau^*+ \epsilon_n$, where $\epsilon_n \geq 0$ and $\epsilon_n \to 0$. For any $n, m \in \N$ let us consider $\nu^t$ a generalized geodesic between $\nu_n$ and $\nu_m$ with base $\mu$ along which $\cG$ is convex. By the convexity of $\cG$ and the $1$-convexity of $W_2^2(\mu, \cdot)$ along this generalized geodesic (see \cite[Remark 9.2.8]{AGS08}), for every $t \in (0,1)$ we obtain
		$$ \cG_{\tau} (\nu^t)  \leq t \cG_{\tau}(\nu_n)+ (1-t) \cG_{\tau}(\nu_m) - \frac { t(1-t)}{2\tau} W_2^2(\nu_n, \nu_m).$$
		Using that $\cG_{\tau}(\nu^{1/2}) \geq \cG_\tau^*$ we get $W_2^2(\nu_n, \nu_m) \leq 4 \tau ( \epsilon_n + \epsilon_m) $, and so $\nu_n$ is a Cauchy sequence in $(\cP_2(\R^d),W_2)$. Thus there exists $\nu_*$ such that $\nu_n \to \nu_*$ in $W_2$. By the lower semicontinuity of $\cG$ we conclude that $\cG_{\tau}(\nu_*)\leq \liminf_n \cG_{\tau} (\nu_n) = \cG_\tau^*$. A similar calculation shows that the minimizer is unique.
		\item[(ii)] Let us consider a generalized geodesic $\nu^t$ between $J_{\tau}(\mu)$ and $\nu$ with base $\mu$. For every $t\in (0,1)$, we have
		$$ \cG_{\tau} (\nu^t)  \leq t \cG_{\tau}(J_{\tau}(\mu))+ (1-t) \cG_{\tau}(\nu) - \frac { t(1-t)}{2\tau} W_2^2(J_{\tau}(\mu), \nu);$$
		using that $\cG_{\tau}(\nu^t) \geq \cG_\tau^* = \cG_{\tau}(J_{\tau}(\mu))$ we can write 
		$$ (1- t) \cG_{\tau}(J_{\tau}(\mu))\leq (1-t) \cG_{\tau}(\nu) - \frac { t(1-t)}{2\tau} W_2^2(J_{\tau}(\mu), \nu).$$
		Dividing by $(1-t)$ and then letting $t \to 1$ we obtain the desired estimate.
		\item[(iii)] For every $\mu$, we denote by $\cG_\tau^*(\mu)$ the value $\inf_{\nu\in\cP_2(\R^d)} \cG_{\tau}(\mu;\nu) := \cG(\nu) + \frac{1}{2\tau} W_2^2(\mu,\nu) $, where now we write explicitely the dependence on $\mu$. By \cite[Lemma 3.1.2]{AGS08} the functional $\cG_\tau^*$ is continuous on $\cP_2(\R^d)$. We define now $\nu_n:= J_\tau(\mu_n)$ and $\nu:= J_\tau(\mu_n)$. Then, we have
		\[\limsup_n \cG_\tau(\mu;\nu_n) = \limsup_n \cG_\tau(\mu_n;\nu_n) = \lim_n \cG_\tau^*(\mu_n)= \cG_\tau^*(\mu).\]
		This implies that for every $\epsilon > 0$ there exists $N_{\epsilon}$ such that \begin{equation}\label{eq:limsup_Ggamma}
			\cG_\tau(\mu;\nu_n)\leq \cG_\tau^*(\mu) + \epsilon = \cG_\tau(\mu;\nu) + \epsilon,
		\end{equation}
		for every $n \geq N_{\epsilon}$. We notice that $\nu_n\in D(\cG)$ for all $n\in \N$. By \cite[Lemma 9.2.7]{AGS08}, there exists a geodesic $\nu^t$ between $\nu_n$ and $\nu$, with base point $\mu$, such that
		\[\cG_\tau (\mu; \nu^t)\leq (1-t)\cG_\tau (\mu;\nu_n)+t\cG_\tau (\mu;\nu)-\frac{ t(1-t)}{2\tau}W_2^2(\nu_n,\nu).\] 
		From \eqref{eq:limsup_Ggamma} we obtain
		\[\cG_\tau (\mu;\nu^t)\leq \cG_\tau (\mu;\nu)+(1-t)\epsilon- \frac{t(1-t)}{2\tau}W_2^2(\nu_n,\nu).\]
		Using that $\cG_\tau(\mu;\nu)=\cG_\tau^*(\mu)= \inf \cG_\tau(\mu;\cdot)$ we have
		\[t(1-t)W_2^2(\nu_n,\nu) \leq 2\tau(1-t)\epsilon.\]
		Dividing by $(1-t)$ and letting $t\to 1$, we obtain $W_2^2(\nu_n,\nu)\leq 2\tau \epsilon$ and $W_2(\nu_n,\nu)\to 0$.
	\end{itemize}
\end{proof}

\section{Inexact JKO}\label{sec:JKO}

As explained in the introduction, we consider here two different choices of error that we allow to be committed from an iteration to another of the JKO scheme. We consider a proper, lower semicontinuous functional $\cG : \cP_2(\R^d) \to \RR$ which is convex along generalized geodesics and that satisfies $\argmin \cG \neq \emptyset$.

\subsection{Distance-type error}\label{sec:first_error} We consider in this section a sequence  $\{\mu_n\}_n$ generated computing the output $J_{\tau_n}(\mu_n)$ with a certain precision with respect to the Wassertein distance, i.e., a sequence satisfying the following
\begin{equation*}
	W_2(\mu_{n+1},J_{\tau_n}(\mu_n))\leq \epsilon_n \quad \text{for all } n\in \N,
\end{equation*}
where $\{\epsilon_n\}_n\subset \R_{\geq 0}$ and $\{\tau_n\}_n\subset \R_{> 0}$. From this, it follows from the triangular inequality that, for every $\nu\in \cP_2(\R^d)$ and every $n\in \N$
\[W_2(\mu_{n+1},\nu)\leq W_2(J_{\tau_n}(\mu_n),\nu)+W_2(\mu_{n+1},J_{\tau_n}(\mu_n)),\]
so that
\begin{equation*}
	W_2(\mu_{n+1},\nu) \leq W_2(J_{\tau_n}(\mu_n),\nu) + \epsilon_n \quad \text{for all } \nu\in \cP_2(\R^d).
\end{equation*}

\begin{lemma}\label{lem:bounded_sequences}
	Suppose that $\sum_n \epsilon_n < +\infty$. Then for every $\nu \in \argmin \cG$ the sequences $\{W_2(\mu_n,\nu)\}_n$ and $\{W_2(J_{\tau_n}(\mu_n),\nu)\}_n$ converge and there exists a contant $C>0$ such that
	\begin{equation}\label{eq:square_inequality}
		W_2^2(\mu_{n+1},\nu)\leq W_2^2(J_{\tau_n}(\mu_n),\nu)+C\epsilon_n.
	\end{equation}
\end{lemma}
\begin{proof}
	For every $\nu \in \argmin \cG$, we have from \eqref{eq:discrete_EVI} \[W_2(J_{\tau_n}(\mu_n),\nu)\leq  W_2(\mu_n,\nu),\] so that
	\[W_2(\mu_{n+1},\nu)\leq W_2(\mu_n,\nu) + \epsilon_n,\]
	and the sequence $\{W_2(\mu_n,\nu)\}_n$ converges. On the other hand, we also have
	\[W_2(J_{\tau_n}(\mu_{n+1}),\nu)\leq W_2(\mu_{n+1},\nu)\leq W_2(J(\mu_{n}),\nu)+\epsilon_{n},\]
	which shows that $\{W_2(J_{\tau_n}(\mu_n),\nu)\}_n$ converges too. It also follows that there exists a $c>0$ such that $W_2(\mu_n,\nu)\leq c $ for all $n\in \N$ (and thus also $W_2(J_{\tau_n}(\mu_n),\nu)\leq c $ for all $n\in \N$).\\
	Thus, we obtain
	\begin{equation*}
		\begin{aligned}
			W_2^2(\mu_{n+1},\nu) & = W_2(\mu_{n+1},\nu)W_2(\mu_{n+1},\nu)\\
			&\leq (W_2(J_{\tau_n}(\mu_n),\nu) +\epsilon_n)(W_2(J_{\tau_n}(\mu_n),\nu) +\epsilon_n)\\
			& \leq W_2^2(J_{\tau_n}(\mu_n),\nu) + 2 c \epsilon_n + \epsilon_n^2,
		\end{aligned}
	\end{equation*}
	and since $\epsilon_n$ is bounded, this concludes the proof.
\end{proof}

The previous lemma implies quasi Fejer monotonicity with respect to $\argmin \cG$ (see \cite[Definition 5.32]{BCombettes} for the definition in Hilbert spaces). Specifically, it follows that for all $\nu \in \argmin \cG$ there exists a constant $C > 0$ such that
\begin{equation}
	W_2^2(\mu_{n+1},\nu)\leq W_2^2( \mu_n,\nu)+C\epsilon_n.
\end{equation}
From now on, we will use the notation $\tilde{\epsilon}_n := C \epsilon_n$, for all $n\in \N$.
\begin{lemma}\label{lem:JKO_asymptotic_regularity}
	It holds that $W_2(J_{\tau_n}(\mu_n),\mu_n)\to 0$ as $n\to +\infty$.
\end{lemma}
\begin{proof}
	Combining \eqref{eq:square_inequality} with the discrete EVI inequality \eqref{eq:discrete_EVI}, we obtain
	\begin{equation}\label{eq:inexact_discrete_EVI}
		W_2^2(\mu_{n+1},\nu) - W_2^2(\mu_n,\nu) \leq 2\tau_n \left(\cG(\nu) -  \cG(J_{\tau_n}(\mu_n))\right) - W_2^2(J_{\tau_n}(\mu_n),\mu_n) +\tilde{\epsilon}_n,
	\end{equation}
	for all $\nu \in \argmin \cG$. From this, we derive
	\begin{equation*}
		W_2^2(\mu_{n+1},\nu) \leq  W_2^2(\mu_n,\nu) -  W_2^2(J_{\tau_n}(\mu_n),\mu_n) +\tilde{\epsilon}_n.
	\end{equation*}
	Since $\sum_n \tilde{\epsilon}_n = C \cdot \sum_n \epsilon_n <+\infty$, summing up both sides gives
	\begin{equation}\label{eq1}
		\sum_n W_2^2(J_{\tau_n}(\mu_n),\mu_n)<+\infty.
	\end{equation}
	This implies in particular that $W_2(J_{\tau_n}(\mu_n),\mu_n)\to 0$.
\end{proof}

\begin{theorem}\label{thm:JKO_convergence1}
	Let $\cG : \cP_2(\R^d) \to \RR $ be proper, lower semicontinuous, and convex along generalized geodesics, with $\argmin \cG \neq \emptyset$. Let $\{\epsilon_n\}_n\subset \mathbb{R}_{\geq 0}$ with $\sum_{n=0}^{\infty} \epsilon_n < \infty$ and let $\{\tau_n\}_n \subset \mathbb{R}_{> 0}$ with $\sum_{i=0}^{\infty}\tau_i= \infty$. Define $\sigma_n:= \sum_{i=0}^{n-1}\tau_i$, for $n\in\mathbb{N}$. Let $\{\mu_n\}_n$ be a sequence satisfying
	\[W_2(\mu_{n+1},J_{\tau_n}(\mu_n))\leq \epsilon_n, \quad \text{for all } n \in \N.\]
	\begin{enumerate}
		\item It holds the rate
		\begin{equation}\label{eq:rate_bar1}
			\cG(\bar \beta_n)-\inf \cG = O\left(\frac{1}{\sigma_n}\right), \quad \text{as } n\to \infty,
		\end{equation}
		where $\bar \beta_n := J_{\tau_{j_n}}(\mu_{j_n})$ with $j_n = \argmin_{i=0,\dots,n-1}\{\cG(J_{\tau_i}(\mu_i))\}$, defines the sequence of the best iterates.
		\item If $\sum_{n=1}^{\infty}\frac{\sigma_n}{\tau_n}\epsilon_{n-1}^2  < \infty$, then \[\cG(J_{\tau_n}(\mu_n))-\inf \cG = O\left(\frac{1}{\sigma_n}\right), \quad \text{as } n\to \infty,\] and $\{\mu_n\}_n$ converges with respect to the topology $\tau_{w,2}$ to some $\mu^*\in \argmin \cG$. In particular we have $W_p(\mu_k,\mu^*)\to 0$ for all $p\in[1,2)$.
	\end{enumerate}
\end{theorem}
\begin{proof}
	We first prove that every weak cluster point of $\{J_{\tau_n}(\mu_n)\}_n$ is a minimizer of $\cG$. We recall equation \eqref{eq:inexact_discrete_EVI} from which we know it holds
	\begin{equation}\label{eq:inexact_discrete_EVI_2}
		W_2^2(\mu_{n+1},\nu) -  W_2^2(\mu_n,\nu) \leq 2\tau_n(\cG(\nu) -  \cG(J_{\tau_n}(\mu_n))) -  W_2^2(J_{\tau_n}(\mu_n),\mu_n) +\tilde{\epsilon}_n,
	\end{equation}
	for all $\nu \in \argmin \cG$ and $n \in \N$. Summing up from $0$ to $N-1$, we obtain
	\begin{equation}\label{eq:for_rates1}
		2\sum_{n=0}^{N-1}\tau_n\cG(J_{\tau_n}(\mu_n))+W_2^2(\mu_N,\nu)\leq 2\sigma_N \cG(\nu)+W_2^2(\mu^0,\nu)+\sum_{n=0}^{N-1}(\tilde{\epsilon}_n-W_2^2(J_{\tau_n}(\mu_n),\mu_n)).
	\end{equation}
	Since $\bar \beta_N$ is the best iterate, we have
	\begin{equation}\label{eq:bar_convexity}
		\cG(\bar \beta_N)\leq \frac{1}{\sigma_N}\sum_{n=0}^{N-1}\tau_n\cG(\mu_{n+1}),
	\end{equation}
	and combining this with equation \eqref{eq:for_rates1}, we obtain
	\begin{equation}\label{eq:rate_ergodic1}
		\cG(\bar \beta_N) - \inf \cG \leq \frac{C}{\sigma_N}.
	\end{equation}
	The first part of the theorem is established. On the other hand, for all $n\geq 1$, we have
	\begin{equation}\label{eq:decrease_JKO}
		\begin{aligned}
			\cG(J_{\tau_n}(\mu_{n}))+\frac{1}{2\tau_n}W_2^2(J_{\tau_n}(\mu_{n}),\mu_{n}) & \leq \cG(J_{\tau_{n-1}}(\mu_{n-1}))+\frac{1}{2\tau_n}W_2^2(J_{\tau_{n-1}}(\mu_{n-1}),\mu_{n})\\ &\leq \cG(J_{\tau_{n-1}}(\mu_{n-1}))+\frac{1}{2\tau_n}\epsilon_{n-1}^2,
		\end{aligned}
	\end{equation}
	and thus, multiplying by $\sigma_n$
	\begin{equation}
		\begin{aligned}
			(\sigma_{n+1}-\tau_n)\cG(J_{\tau_n}(\mu_{n})) - \sigma_{n}\cG(J_{\tau_{n-1}}(\mu_{n-1})) \leq \frac{\sigma_n}{2\tau_n}\epsilon_{n-1}^2.
		\end{aligned}
	\end{equation}
	Summing from $1$ to $N-1$ and recalling that $\sigma_1=\tau_0$, we obtain
	\begin{equation}\label{eq:inexact_decreasing_sum}
		\begin{aligned}
			\sigma_{N}\cG(J_{\tau_{N-1}}(\mu_{N-1})) - \sum_{n=1}^{N-1}\frac{\sigma_n}{2\tau_n}\epsilon_{n-1}^2 \leq \sum_{n=0}^{N-1}\tau_n \cG(J_{\tau_n}(\mu_{n})).
		\end{aligned}
	\end{equation}
	Combining \eqref{eq:for_rates1} and \eqref{eq:inexact_decreasing_sum}, we arrive at
	\begin{equation}\label{eq:for_rates2}
		2\sigma_{N}\cG(J_{\tau_{N-1}}(\mu_{N-1})) - 2\sigma_N \cG(\nu)\leq W_2^2(\mu^0,\nu)+\sum_{n=0}^{N-1}\tilde{\epsilon}_n + \sum_{n=1}^{N-1}\frac{\sigma_n}{\tau_n}\epsilon_{n-1}^2.
	\end{equation}
	Since by hypothesis $\tilde{\epsilon}_n$ and  $\frac{\sigma_n}{\tau_n}\epsilon_{n-1}^2$ are summable, we get
	\begin{equation}\label{eq:rate_Jmun}
		\cG(J_{\tau_{N-1}}(\mu_{N-1})) - \inf \cG \leq \frac{C}{\sigma_N},
	\end{equation}
	for some constant $C>0$. To conclude we use the Opial property (\Cref{thm:Opial}) to prove convergence of the whole sequence. First we notice that $\{J_{\tau_n}(\mu_n)\}_n$ is bounded in $(\cP_2(\R^d),W_2)$ and thus \[\sup_{n\in \N} \int \|x\|^2 \, d J_{\tau_n}(\mu_n)(x)= \sup_{n\in \N} W_2^2(J_{\tau_n}(\mu_n),\delta_0) < +\infty.\]
	Using for example \cite[Corollary 3.6 (c)]{NS2021} we conclude that $\{J_{\tau_n}(\mu_n)\}_n$ has at least a cluster point with respect to the topology $\tau_{w,2}$ in $\cP_2(\R^d)$.
	Consider a subsequence $\{J_{\tau_{n_i}}(\mu_{n_i})\}_i$ of $\{J_{\tau_n}(\mu_n)\}_n$ and a cluster point $\mu\in \cP(X)$ such that $J_{\tau_{n_i}}(\mu_{n_i})\wsto \mu$, then, by lower semicontinuity of $\cG$ (also with respect to the topology $\tau_{w,2}$ we consider on $\mathcal{P}^2(\R^d)$, see \Cref{thm:lsc_wlsc}), we obtain
	\[\cG(\mu)\leq \liminf_i \cG(J_{\tau_{n_i}}(\mu_{n_i}))\leq \limsup_i \cG(J_{\tau_{n_i}}(\mu_{n_i})) \leq \inf \cG,\]
	which means that $\mu \in \argmin \cG$. This proves that every cluster point of the sequence is a minimizer of $\cG$.\\
	Now, in order to prove convergence of the whole sequence, we suppose there exist two subequences $\{J_{\tau_{n_i}}(\mu_{n_i})\}_i$ and $\{J_{\tau_{m_i}}(\mu_{m_i})\}_i$ of $\{J_{\tau_n}(\mu_n)\}_n$ with $J_{\tau_{n_i}}(\mu_{n_i})\wsto\nu^*\in \cP_2(\R^d)$ and $J_{\tau_{m_i}}(\mu_{m_i})\wsto\nu^{**}\in \cP_2(\R^d)$. We define $\ell(\nu):= \lim_n W_2(J_{\tau_n}(\mu_n),\nu)$ for all $\nu \in \argmin \cG$, which always exists by \Cref{lem:bounded_sequences}. Since we have proven $\nu^*, \nu^{**}\in \argmin \cG$, we can use the Opial property and obtain whenever $\nu^*\neq\nu^{**}$
	\begin{equation*}
		\begin{aligned}
			\ell(\nu^*) = \liminf_i W_2(J_{\tau_{n_i}}(\mu_{n_i}),\nu^*) < \liminf_i W_2^2(J_{\tau_{n_i}}(\mu_{n_i}),\nu^{**})=\ell(\nu^{**})\\
			\ell(\nu^{**}) = \liminf_i W_2(J_{\tau_{m_i}}(\mu_{m_i}),\nu^{**}) < \liminf_i W_2^2(J_{\tau_{m_i}}(\mu_{m_i}),\nu^{*})=\ell(\nu^{*})
		\end{aligned}
	\end{equation*}
	so that it must be $\nu^*=\nu^{**}$. We have proved that there exists $\mu^*\in \argmin \cG$ such that $J_{\tau_n}(\mu_n)\wsto \mu^*$. Since both $\{J_{\tau_n}(\mu_n)\}_n$ and $\{\mu_n\}_n$ are bounded sequences and by \Cref{lem:JKO_asymptotic_regularity} we have $W_2(J_{\tau_n}(\mu_n),\mu_n)\to 0$, we can conclude that $\mu_n\wsto \mu^*$ using \Cref{lem:topology}.
\end{proof}

\begin{remark}\label{rem:strongly-convex} Notice that imposing that $\cG$ is $\lambda_{\cG}$-convex along generalized geodesics \cite[Definition 9.2.4]{AGS08}, with $\lambda_{\cG} >0$, an analogue of \cite[Theorem 4.1.2 (ii)]{AGS08} combined with \eqref{eq:inexact_discrete_EVI_2} yields the inequality \[(1+\lambda_{\cG} \tau_n) W_2^2(\mu_{n+1},\nu)\leq W_2^2(\mu_n,\nu) + \tilde{\epsilon}_n\, \quad \text{ for all } \nu \in \argmin \cG.\] From this, strong convergence  in $W_2$ can be obtained. However, due to the error $\epsilon_n$ incurred at each iteration, the convergence rate cannot be improved without further assumptions on $\{\epsilon_n\}_n$. This consideration motivates our decision not to treat the strongly convex case separately.
\end{remark}

\begin{remark}
	The condition $\sum_{n=0}^{\infty}\frac{\sigma_n}{\tau_n}\epsilon_{n-1}^2  < \infty$ is not overly restrictive and it is somehow what we expect. This condition is satisfied in particular in the following cases:
	\begin{itemize}
		\item $\epsilon_n=0$ for all $n\in \mathbb{N}$
		\item $\{\epsilon_n\}_n$ is nonincreasing and $0 < \inf_n \tau_n \leq \sup_n \tau_n < M$ for some $M>0$.
		\item $\{\epsilon_n\}_n$ and $\{\tau_n\}_n$ nonincreasing and $\sigma_n\epsilon_{n-1}/\tau_n$ is bounded, for example with $\epsilon_n=\frac{1}{n^{1+\delta}}$ and $\tau_n= \frac{1}{n}$, for all $n\in \N$, where $\delta > 0$.
		\item $\{\epsilon_n\}_n$ is nonincreasing and $\{\tau_n\}_n$ nondecreasing.
	\end{itemize}
\end{remark}

\begin{remark}[Ergodic convergence]\label{rmk:ergodic}
		Supposing that the elements of the sequence $\{J_{\tau_n}(\mu_n)\}_n$ belong to $\cP_2^r(\R^d)$  we can apply \cite[Proposition 7.6]{Carlier11} which states that if a functional is convex along generalized geodesics then it is convex along barycenters (for regular measures). In particular, this implies that  \eqref{eq:bar_convexity} holds for $\bar \beta_n:= \text{Bar}\left(J_{\tau_i}(\mu_i), \frac{\tau_i}{\sigma_n}\right)_{i=0,\dots,n-1}$, the Wasserstein barycenter of the first $n$ elements of the sequence~$\{J_{\tau_i}(\mu_i)\}_i$ with parameters $\{\frac{\tau_i}{\sigma_n}\}_0^{n-1}$. As a result, \eqref{eq:rate_bar1} holds for the barycenter sequence $\{\bar \beta_n\}_n$, demonstrating that ergodic-type convergence rates can be achieved without requiring the additional assumptions in point~2 of \Cref{thm:JKO_convergence1}. We expect this result to hold under more general assumptions, for example for functionals for which regular measures are dense in energy. In such cases, one could exploit the stability of Wasserstein barycenters, see \cite[Theorem 3]{LeGouic2017} and \cite{Carlier2024}. We do not pursue this direction since Wasserstein barycenters are not as practical to compute as their Hilbertian counterpart, thereby limiting their practical relevance.
	\end{remark}
	
	\begin{remark}
		When $\cG$ is $L_{\cG}$-Lipschitz continuous in a ball that contains $\{\mu_n\}_n$ and $\{J_{\tau_n}(\mu_n)\}_n$, we can derive a rate on $\cG(\mu_n)$. In fact, we have $\cG(\mu_{n+1})\leq \cG(J_{\tau_n}(\mu_n))+ L_{\cG} \epsilon_n$, and if $\epsilon_n = O\left(\frac{1}{\sigma_n}\right)$, for $n\to \infty$, then
		\[\cG(\mu_{n+1})-\inf \cG\leq \cG(J_{\tau_n}(\mu_n))-\inf \cG+ L_{\cG}\epsilon_n = O\left(\frac{1}{\sigma_n}\right), \quad \text{for } n\to \infty.\]
		Notice, however, that in Wasserstein spaces, the connection between convexity and local Lipschitz continuity is not obvious. As an example, let $\cG: \cP_2(\R^d) \to \RR$ be proper, lower semicontinuous and convex along geodesics. If $\cG$ is approximable by discrete measures and the domain of $\cG$ is totally convex, then $\cG$ is locally Lipschitz, see \cite[Remark 9]{CSS_2023}. However, several functionals are not locally Lipschitz continuous. In particular, if the domain of $\cG$ is included in the set of regular measures $\cP_2^r(\R^d)$, then $\cG$ cannot be locally Lipschitz continuous since every regular measure can be approximated in $(\cP_2(\R^d), W_2)$ by discrete measures. A notable example where locally Lipschitz continuity fails is given by the negative entropy.
\end{remark}

\paragraph{Nonexpansivity of the proximal map}

In classical analysis in Hilbert spaces, the weak convergence of the sequence $\{\mu_n\}_n$ is established in a more direct way. However, such an analysis relies on the nonexpansivity of the proximal map, a property that is not known to hold in Wasserstein spaces. Indeed, this remains an open problem. Even the projection onto convex sets can fail to be convex. As a first example, consider the set of two-atomic measures \[S=\{\mu\in \cP_2(\R^d)\mid \#\operatorname{supp}(\mu)=2\}.\] This set is closed and geodesically convex, but the projection onto this set fails to be Lipschitz, as also shown in \Cref{fig:projection_failure}.
\begin{figure}[h]
	\centering
	\begin{subfigure}[b]{0.49\textwidth}
		\includegraphics[width=\textwidth]{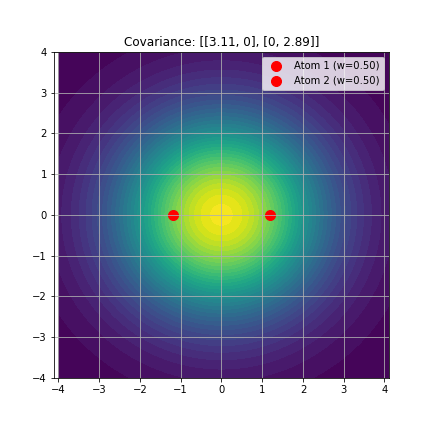}
	\end{subfigure}
	\hfill
	\begin{subfigure}[b]{0.49\textwidth}
		\includegraphics[width=\textwidth]{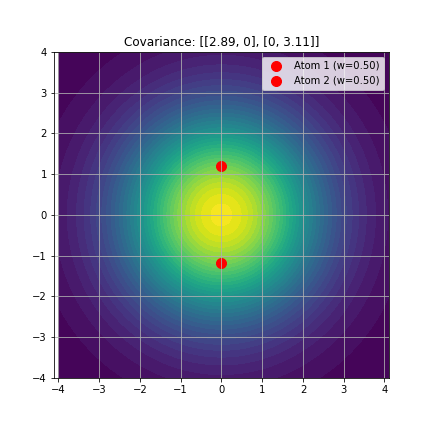}
	\end{subfigure}
	\caption{Projection of two nearby 2D Gaussians onto the set $S$. Despite having really similar distributions, the projections differ significantly, showing that the projection map is not Lipschitz.}
	\label{fig:projection_failure}
\end{figure}

Even considering sets which are convex along generalized geodesics, the Lipschitzianity could fail or remain open, as commented in \cite{DePhilippis2016}. Consider the set
\[S' := \left\{\rho \in L_1^{+}(\R^d) \mid \int \rho(x) \, dx = 1, \ \rho \leq 1\right\}\cap \cP_2(\R^d)\]
In \cite[Corollary 5.3]{DePhilippis2016} it is proven that the projection onto $S'$ is locally $\frac{1}{2}$-H\"older, but in \cite[Remark 5.1]{DePhilippis2016} the authors say that Lipschitzianity is still an open problem.\\
For positive results, some insights are provided in \cite{CarlenCraig2013}, where a slight modification of the metric is considered and there are specific cases where the nonexpansivity of $J_{\tau}$ is guaranteed. A known positive example of the nonexpansivity of the proximal map, is when the functional $\cG$ is totally convex, i.e., convex along any coupling, see \cite[Definition 2.7]{CSS_2023}. In this case, it is possible to introduce a probability space $(\Omega, \mathcal{B},\mathbb{P})$, consider the function \[\begin{aligned}
g: L^2(\Omega, \mathbb{P};\mathbb{R}^d) & \to (-\infty, +\infty] \\
X & \mapsto \mathcal{G}( X_{\#} \mathbb{P} )\end{aligned}\] and prove that $g$ is proper, convex and lower semicontinuous. By the last part of \cite[Proposition 5.2 (1)]{CSS_2023}, whenever $\mu = X_{\#}\mathbb{P}$, then $J_{\tau}(\mu)=(\prox_{\tau g}^{L^2} \circ X)_{\#}\mathbb{P}$. Since $g$ is proper, convex and lower semicontinuous, then $\prox_{\tau g}^{L^2}:L^2(\Omega,\mathbb{P};\R^d)\to L^2(\Omega,\mathbb{P};\R^d)$ is nonexpansive \cite[Proposition 12.28]{BCombettes}. Thus, given $\mu, \mu'\in \cP_2(\R^d)$, for every $X,X'$ such that $X_{\#}\mathbb{P}= \mu$, $X'_{\#}\mathbb{P}= \mu'$, we have
\[W_2^2(J_{\tau}(\mu),J_{\tau}(\mu'))\leq \|\prox_{\tau g}^{L^2}(X)-\prox_{\tau g}^{L^2}(X')\|_{L^2}^2\leq \|X-X'\|_{L^2}^2.\]
So that
\[W_2(J_{\tau}(\mu),J_{\tau}(\mu'))\leq \inf_{\substack{X \sim \mu \\ X' \sim \mu'}} \|X - X'\|_{L^2}^2 = W_2(\mu,\mu').\]

\subsection{Variational-type error}\label{sec:second_error}

Considering the second error choice described in the introduction, we can prove the following result.

\begin{theorem}\label{thm:second_error_choice}
	Let $\cG : \cP_2(\R^d) \to \RR $ proper, lower semicontinuous and convex along generalized geodesics and suppose $\argmin \cG \neq \emptyset$. Let $\{\epsilon_n\}_n\subset \mathbb{R}_{\geq 0}$ with $\sum_{n=0}^{\infty} \epsilon_n < \infty$ and let $\{\tau_n\}_n \subset \mathbb{R}_{> 0}$ with $\sum_{i=0}^{\infty}\tau_i= \infty$. Define $\sigma_n:= \sum_{i=0}^{n-1}\tau_i$, for $n\in\mathbb{N}$. Consider $\{\mu_n\}_n$ such that for all $n\in \N$ it holds
	\begin{equation}\label{eq:second_error_condition}
		\cG(\mu_{n+1})+\frac{1}{2{\tau_n}}W_2^2(\mu_{n+1},\mu_n)\leq \cG(J_{\tau_n}(\mu_n))+\frac{1}{2\tau_n}W_2^2(J_{\tau_n}(\mu_n),\mu_n)+ \frac{\epsilon_n^2}{2\tau_n}.
	\end{equation}
	\begin{enumerate}
		\item It holds the rate
		\begin{equation*}
			\cG(\beta_n)-\inf \cG = O\left(\frac{1}{\sigma_n}\right), \quad \text{as } n\to \infty,
		\end{equation*}
		where $\beta_n := \mu_{j_n}$ with $j_n = \argmin_{i=0,\dots,n-1}\{\cG(\mu_i)\}$, defines the sequence of the best iterate.
		\item If $\sum_{n=0}^{\infty}\frac{\sigma_n}{\tau_n}\epsilon_{n}^2  < \infty$, then \[\cG(\mu_n)-\inf \cG = O\left(\frac{1}{\sigma_n}\right), \quad \text{as } n\to \infty,\] and $\{\mu_n\}_n$ converges with respect to the topology $\tau_{w,2}$ to some $\mu^*\in \argmin \cG$. In particular we have $W_p(\mu_k,\mu^*)\to 0$ for all $p\in[1,2)$.
	\end{enumerate}
\end{theorem}
\begin{proof}
	Set $\cG_{{\tau_n}}(\mu):= \cG(\mu)+\frac{1}{2{\tau_n}}W_2^2(\mu,\mu_n)$. The assumption \eqref{eq:second_error_condition} then reads
	\begin{equation}\label{eq:condition2}
		\cG_{{\tau_n}}(\mu_{n+1})\leq \cG_{\tau_n}(J_{{\tau_n}} (\mu_n)) +\frac{\epsilon_n^2}{2\tau_n}.
	\end{equation}
	By \cite[Lemma 9.2.7]{AGS08}, there exists a generalized geodesic $\nu^t$ between $\mu_{n+1}$ and $J_{{\tau_n}}(\mu_n)$, with base point $\mu_n$, such that
	\[\cG_{\tau_n} (\nu^t)\leq (1-t)\cG_{\tau_n} (\mu_{n+1})+t\cG_{\tau_n} (J_{{\tau_n}}(\mu_n))-\frac{ t(1-t)}{2{\tau_n}}W_2^2(\mu_{n+1},J_{{\tau_n}}(\mu_n)).\] 
	From \eqref{eq:condition2} we derive
	\[\cG_{\tau_n} (\nu^t)\leq \cG_{\tau_n} (J_{{\tau_n}}(\mu_n))+(1-t)\frac{2\epsilon_n^2}{{\tau_n}}-\frac{t(1-t)}{2{\tau_n}}W_2^2(\mu_{n+1},J_{{\tau_n}}(\mu_n))\]
	and
	\[W_2^2(\mu_{n+1},J_{{\tau_n}}(\mu_n)) \leq 2{\tau_n} \frac{\cG_{\tau_n} (J_{{\tau_n}}(\mu_n))-\cG_{\tau_n} (\nu^t)+(1-t)\frac{\epsilon_n^2}{2\tau_n}}{t(1-t)}.\]
	Using that $\cG_{\tau_n}(J_{{\tau_n}}(\mu_n))= \inf_\mu \cG_{\tau_n}(\mu)$ and letting $t\to 1$, we obtain
	\[W_2^2(\mu_{n+1},J_{{\tau_n}}(\mu_n))\leq \epsilon_n^2 \quad \text{ and }\quad W_2(\mu_{n+1},J_{\tau}(\mu_n))\leq \epsilon_n.\]
	We can then apply \Cref{thm:JKO_convergence1} to obtain convergence of the sequence $\{\mu_n\}_n$ to a minimizer of $\cG$. In order to derive the rate on $\cG(\mu_n)$, we proceed as follows. We first recall from \eqref{eq:for_rates1} that
	\begin{equation}\label{eq:err2_for_rates1}
		2\sum_{n=0}^{N-1}\tau_n\cG(J_{\tau_n}(\mu_n))+W_2^2(\mu_N,\nu)\leq 2\sigma_N \cG(\nu)+W_2^2(\mu^0,\nu)+\sum_{n=0}^{N-1}(\tilde{\epsilon}_n-W_2^2(J_{\tau_n}(\mu_n),\mu_n)).
	\end{equation}
	and thus, assumption \eqref{eq:second_error_condition} implies
	\begin{equation}\label{eq:err2_for_rates2}
		2\sum_{n=0}^{N-1}\tau_n\cG(\mu_{n+1})+W_2^2(\mu_N,\nu)\leq 2\sigma_N \cG(\nu)+W_2^2(\mu^0,\nu)+\sum_{n=0}^{N-1}\tilde{\epsilon}_n+ \sum_{n=0}^{N-1}\epsilon_n^2,
	\end{equation}
	and point~1 follows from the same reasoning done in the proof of \Cref{thm:JKO_convergence1}. On the other hand, \eqref{eq:second_error_condition} yields
	\[\mathcal{G}(\mu_{n+1})\leq \cG(J_{{\tau_n}}(\mu_n))+\frac{1}{2\tau_n}W_2^2(J_{\tau_n}(\mu_n),\mu_n)+ \frac{\epsilon_n^2}{2\tau_n}\leq  \mathcal{G}(\mu_n)+\frac{\epsilon_n^2}{2\tau_n}.\]
	Reasoning as in the proof of \Cref{thm:JKO_convergence1}, we multiply by $\sigma_n$ and obtain
	\begin{equation}
		\begin{aligned}
			(\sigma_{n+1}-\tau_n)\cG(\mu_{n+1}) - \sigma_{n}\cG(\mu_n) \leq \frac{\sigma_n}{2\tau_n}\epsilon_{n}^2.
		\end{aligned}
	\end{equation}
	Summing from $0$ to $N-1$ and recalling that $\sigma_0=0$, we derive
	\begin{equation}
		\begin{aligned}
			\sigma_{N}\cG(\mu_N) - \sum_{n=0}^{N-1}\frac{\sigma_n}{2\tau_n}\epsilon_{n}^2 \leq \sum_{n=0}^{N-1}\tau_n \cG(\mu_{n+1}).
		\end{aligned}
	\end{equation}
	Combining, we obtain
	\begin{equation}\label{eq:err2_for_rates3}
		2\sigma_N \mathcal{G}(\mu_N)\leq 2\sigma_N \cG(\nu)+W_2^2(\mu^0,\nu)+\sum_{n=0}^{N-1}\tilde{\epsilon}_n+\sum_{n=0}^{N-1}\epsilon_n^2+\sum_{n=0}^{N-1}\frac{\sigma_n}{\tau_n}\epsilon_{n}^2.
	\end{equation}
	Using the summability conditions, we end up with
	\begin{equation}\label{eq:rate_mun}
		\cG(\mu_N) - \inf \cG \leq \frac{C}{\sigma_N}.
	\end{equation}
\end{proof}

\paragraph{Obtaining error bounds approximating the solution}
Optimization schemes that tackle problem \eqref{eq:prox_point_intro} usually make use of the Benamou-Brenier formula \cite{BB2000} to rewrite the $2$-Wasserstein distance, and solve a saddle point problem. A well-known method to approximate a solution of \eqref{eq:prox_point_intro} is the so-called ``ALG2" introduced in \cite{BCL16}. The algorithm makes use of an augmented Lagrangian method called alternating direction method
of multipliers (ADMM). The ADMM method has a long history and presents solid guarantees of convergence \cite{ADMM_powell69,ADMM_hestenes69,ADMM_gabay76,ADMM_gabay83,ADMM_fortin83,ADMM_glowinski87}. For example, in the overview of Boyd \cite{ADMM_boyd11} the convergence of the ``dual" variable is stated in \cite[Section 3.2.1]{ADMM_boyd11}. Using the notations of \cite{BCL16}, the ``dual" variable correspond to $\sigma = (\mu, m, \mu_1)$. In particular, the convergence of $\sigma_n$ to the optimal dual variable $\sigma^*$ implies the convergence of $\mu_{1,n}$ to the solution $\mu_1^*$ to problem \eqref{eq:prox_point_intro} (or \cite[Problem (1.4)]{BCL16}). It is clear that both $F$ and $G$ in \cite{BCL16} are not strongly convex, independently on the choice of the initial functional to minimize (in our notations $\cG$). Thus, it's not obvious how to guarantee linear convergence for the ``dual" variable $\sigma$. Energy convergence is also stated in \cite[Section 3.2.1]{ADMM_boyd11}, but it is not clear how to find a condition of the kind \eqref{eq:second_error_condition}. It is possible that results about convergence in energy of the closely related Douglas-Rachford method or its generalization have to be used \cite{EB92, BrediesDRS, Davis2016, Davis_rates, Bredies24}. This can be the subject of further investigation.

\bigskip

\paragraph{Obtaining error bounds modifying the problem} As discussed in the introduction, another common practice is to perturb the original problem \eqref{eq:prox_point_intro} to find in an easier way approximated solutions. In \cite{CDPS17} a perturbation of the $W_2$-distance coming from entropic optimal transport \cite{Cuturi2013} is used instead of $W_2$ in \eqref{eq:prox_point_intro}. However, the authors do not quantify the $W_2$-distance between the output of the entropic proximal step and the output of the classic proximal step, so it is not clear if it is possible to achieve one of the previous conditions, even sending the regularization parameter to zero. However, in these cases, it is probably better to try and prove some sort of convergence result for the scheme itself, without relying on how close it is to a classical JKO scheme. We do not expand on this in the current work. Other approximations of the JKO step can be found in \cite{LLO23, LLW20, LWL24}. Also here the authors do not provide any estimate about how far the approximated solutions are from the true solution and investigate the quality of the approximations only empirically.

\section{Inexact proximal-gradient algorithm}\label{sec:PG}
As anticipated in the introduction, in this section we focus on the problem
	\begin{equation}
		\label{eq:psum}
		\min_{\mu\in\cP_2(\R^d)} \cG (\mu) = \mathcal{E}_F(\mu)+\mathcal{H}(\mu)
	\end{equation}
	with $\mathcal{E}_F(\mu)= \int F \, \mathrm{d}\mu$. Our analysis is grounded on the following assumptions.
	
	\smallskip
	
	\begin{itemize}
		\item[\textbf{(A1)}]\label{A1} $F\colon\R^d\to\R$ is proper, convex, differentiable with $L$-Lipschitz continuous gradient and $\lambda$-strongly convex (with $\lambda=0$ permitted)
		\item[\textbf{(A2)}] $\mathcal{H}\colon \cP_2(\R^d)\to\R \cup \{+\infty\}$ is proper, lower semicontinuous and convex along generalized geodesics
		\item[\textbf{(A3)}] $\text{dom}(\mathcal{H})\subset \cP_2^r(\mathcal{X})$.
	\end{itemize}
	
	\smallskip
	
	Notice that, as an example, the task of minimizing the free energy functional defined by
	\[\mu \mapsto \int F(x) \, d\mu(x) + \text{Ent}(\mu),\]
	can be cast as \eqref{eq:psum} with $\mathcal{H}= \text{Ent}$, which satisfies assumption (A2)-(A3) since it is proper, lower semicontinuous, convex along generalized geodesics and by definition its domain satisfies $\text{dom}(\text{Ent})\subset \cP_2^r(\mathcal{X})$.
	
	The proximal-gradient algorithm has been studied in \cite{WPG2020}. At each iteration makes use of gradient-descent-type step for the functional $\cE_F$ and of a proximal step for $\mathcal{H}$. The resulting scheme is a more general algorithm than the proximal point method (JKO scheme), and the additional assumption (A3) is the reason why we keep separated the analysis of the two algorithms. In this section, we describe an inexact version of the method. We define the operator \[\cT_{\tau}:=J_{\tau , \calH}\circ (I-\tau \nabla F)_{\#}\] and we start by considering an inexact scheme satisfying
	\begin{equation}\label{eq:FB}
		W_2(\mu_{n+1},\cT_{\tau_n}(\mu_n))\leq \epsilon_n, \quad \text{for all } n \in N,
	\end{equation}
	with a sequence of positive stepsizes $\{\tau_n\}_{n}.$

First, analyzing \cite[Proposition 8]{WPG2020} and its proof, it is possible to show that, whenever $\tau < 1/L$, for all $\mu \in \cP_2^r(\R^d)$ and $\nu\in \cP_2(\R^d)$, it holds
\begin{equation}\label{eq:discrete_EVI_FB}
	W_2^2(\cT_{\tau}(\mu), \nu)\leq (1-\tau \lambda)W_2^2(\mu,\nu)-2\tau (\cG(\cT_{\tau}(\mu))-\cG(\nu)).
\end{equation}
It is actually possible to prove a more refined result, which will be crucial.

\begin{lemma}\label{lem:discreteEVI}
	Let $\mathcal{H} : \cP_2(\R^d) \to \mathbb{R} \ $ and $F$ satisfying (A1)-(A3). Let $\cT_{\tau}:=J_{\tau , \calH}\circ (I-\tau \nabla F)_{\#}$ with $\tau < 1/L$. Then
	\begin{equation}\label{eq:discrete_EVI_FB_complete2}
		W_2^2(\cT_\tau(\mu), \nu)\leq (1-\tau \lambda)W_2^2(\mu,\nu)-2\tau (\cG(\cT_\tau(\mu))-\cG(\nu))-(1-\tau L)W_2^2(\mu,\cT_\tau(\mu)),
	\end{equation}
	for all $\mu \in \cP_2^r(\R^d)$ and $\nu\in \cP_2(\R^d)$.
\end{lemma}
\begin{proof}
	From the proof of \cite[Proposition 8]{WPG2020}, we can see that there exists a strong Fréchet subgradient of $\mathcal{H}$ at $\cT_{\tau}(\mu)$ denoted by $\nabla_W \mathcal{H} (\cT_{\tau}(\mu))$ such that
	\begin{equation*}\label{eq:discrete_EVI_FB_complete1}
		\begin{aligned}
			W_2^2(\cT_\tau(\mu), \nu)\leq (1-\tau \lambda)W_2^2(\mu,\nu)-2\tau & (\cG(\cT_\tau(\mu))-\cG(\nu))\\ &-(1-\tau L)\|\tau \nabla F + \tau \nabla_W\mathcal{H}(\cT_\tau(\mu))\circ X\|_{\mu}^2,
		\end{aligned}
	\end{equation*}
	where $X=T_{\eta}^{\bar \mu^+}\circ (I-\tau \nabla F)$ and $T_{\eta}^{\bar \mu^+}$ is the optimal transport map between $\eta:=(I-\tau \nabla F)_{\#}(\mu)$ and $\bar{\mu}^+:=J_{\tau,\mathcal{H}}(\eta)=\cT_\tau(\mu)$. On the other hand, since $\left(I+\tau \nabla_W \mathcal{H}(\cT_\tau(\mu))\right)$ is the optimal transport map between $\bar \mu^+$ and $\eta$ (see for example \cite[Lemma 3]{WPG2020} or \cite{AGS08}), we have $\left(I+\tau \nabla_W \mathcal{H}(\cT_\tau(\mu))\right)\circ T_{\eta}^{\bar \mu^+}= I$, and
	\begin{equation}
		\begin{aligned}
			\tau \nabla F + \tau \nabla_W & \mathcal{H}(\cT_\tau(\mu))\circ X\\ & = \tau \nabla F + \left(I+\tau \nabla_W \mathcal{H}(\cT_\tau(\mu))\right)\circ T_{\eta}^{\bar \mu^+}\circ (I-\tau \nabla F) - T_{\eta}^{\bar \mu^+}\circ (I-\tau \nabla F)\\
			& = \tau \nabla F + I-\tau \nabla F - T_{\eta}^{\bar \mu^+}\circ (I-\tau \nabla F) = I - T_{\eta}^{\bar \mu^+}\circ (I-\tau \nabla F).
		\end{aligned}
	\end{equation}
	By \cite[Lemma 2]{WPG2020}, $(I-\tau \nabla F)$ is the optimal transport map between $\mu$ and $\eta$ and thus $T_{\eta}^{\bar \mu^+}\circ (I-\tau \nabla F)$ is a transport map between $\mu$ and $\cT_\tau(\mu)$, which implies $\|I - T_{\eta}^{\bar \mu^+}\circ (I-\tau \nabla F)\|_{\mu}^2\geq W_2^2(\mu,\cT_\tau(\mu))$. Therefore, since $1-\tau L > 0$, we obtain \eqref{eq:discrete_EVI_FB_complete2}.
\end{proof}

In our analysis, the input $\mu$ will be only an approximation of the previous iterate, and thus, we would like to remove the condition $\mu \in \cP_2^r(\R^d)$ in the previous lemma. This becomes crucial while considering optimization schemes used to approximate the output of the JKO operator, since the approximated output is usually a discrete measure (if not parametrized otherwise). We thus state and prove now the lemma in its more general form.

\begin{lemma}\label{lem:discreteEVI2}
	Let $\mathcal{H} : \cP_2(\R^d) \to \mathbb{R} \ $ and $F$ satisfying (A1)-(A3). Let $\cT_{\tau}:=J_{\tau , \calH}\circ (I-\tau \nabla F)_{\#}$ with $\tau < 1/L$. Then
	\begin{equation}\label{eq:discrete_EVI_FB2}
		W_2^2(\cT_\tau(\mu), \nu)\leq (1-\tau \lambda)W_2^2(\mu,\nu)-2\tau (\cG(\cT_\tau(\mu))-\cG(\nu))-(1-\tau L)W_2^2(\mu,\cT_\tau(\mu)),
	\end{equation}
	for all $\mu,\nu \in \cP_2(\R^d)$.
\end{lemma}
\begin{proof}
	Let $\mu,\nu\in \cP_2(\R^d)$. We define $\{\mu_n^r\}_n$ as a sequence of regular measures such that $W_2(\mu_n^r,\mu)\to 0$, which always exists (for example smoothing with a convolution). Then, the continuity of $J_{\tau,\mathcal{H}}$ (see \Cref{thm:introduction} (iii)) and the continuity of $(I-\tau \nabla F)_\#$ (which follows by the fact that the operator $I-\tau \nabla F$ is Lipschitz) implies that $W_2(\cT_{\tau}(\mu_n^r),\cT_{\tau}(\mu))\to 0$. We have both
	\[|W_2(\cT_\tau(\mu_n^r),\nu)-W_2(\cT_\tau(\mu),\nu)|\leq W_2(\cT_\tau(\mu_n^r),\cT_\tau(\mu))\to 0,\] and \[|W_2(\mu_n^r,\cT_\tau(\mu_n^r))-W_2(\mu,\cT_\tau(\mu))|\leq W_2(\mu_n^r,\mu)+W_2(\cT_\tau(\mu),\cT_\tau(\mu_n^r))\to 0.\]
	By \Cref{lem:discreteEVI} we have
	\begin{equation*}
		W_2^2(\cT_\tau(\mu_n^r), \nu)\leq (1-\tau \lambda)W_2^2(\mu_n^r,\nu)-2\tau (\cG(\cT_\tau(\mu_n^r))-\cG(\nu))-(1-\tau L)W_2^2(\mu_n^r,\cT_\tau(\mu_n^r)),
	\end{equation*}
	and, since $\cG$ is lower semicontinuous, we can pass to the limit and obtain \eqref{eq:discrete_EVI_FB2}.
\end{proof}

\begin{remark}[Discrete EVI]\label{rem:discreteEVI}
	Using the previous lemma, we can prove a new EVI inequality for the sequence $\{\mu_n\}_n$ generated as described in \eqref{eq:FB}. In fact, for all $\nu \in \cP_2(\R^d)$ it holds
	\begin{equation}\label{eq:discrete_EVI_FB_true}
		\begin{aligned}
			W_2^2(\cT_{\tau_n}(\mu_{n}), \nu)\leq (1-\tau_n \lambda)W_2^2(\mu_n,\nu)-& 2\tau_n (\cG(\cT_{\tau_n}(\mu_{n}))-\cG(\nu))\\
			&-(1-\tau_n L)W_2^2(\mu_n,\cT_{\tau_n}(\mu_{n})).
		\end{aligned}
	\end{equation}
	This extends the result \cite[Proposition 8]{WPG2020}. This finer EVI inequality, coming from inequality \eqref{eq:discrete_EVI_FB2}, is the one we need for our analysis.
\end{remark}

\begin{lemma}
	Suppose that $\{\epsilon_n\}_n$ is a positive summable sequence and the sequence $\{\tau_n\}_n \subset \left(0,\frac{1}{L}\right)$ satisfies $\sup_i \tau_i < \frac{1}{L}$. Then for every $\nu \in \argmin \cG$ the sequences $\{W_2(\mu_n,\nu)\}_n$ and $\{W_2(\cT_{\tau_n}(\mu_n),\nu)\}_n$ converge and there exists a constant $C>0$ such that
	\begin{equation}\label{eq:square_inequality_PG}
		W_2^2(\mu_{n+1},\nu)\leq W_2^2(\cT_{\tau_n}(\mu_n),\nu)+C\epsilon_n.
	\end{equation}
	It also holds $W_2(\cT_{\tau_n}(\mu_n),\mu_n)\to 0$, as $n\to +\infty$.
\end{lemma}
\begin{proof}
	The first part of the proof is similar to the proof of \Cref{lem:bounded_sequences}, and therefore omitted.\\
	Combining \eqref{eq:square_inequality_PG} and the discrete EVI inequality \eqref{eq:discrete_EVI_FB_true}, we obtain
	\begin{equation*}
		W_2^2(\mu_{n+1},\nu) - W_2^2(\mu_n,\nu) \leq - (1-\tau_n L) W_2^2(\cT_{\tau_n}(\mu_n),\mu_n) +\tilde{\epsilon}_n,
	\end{equation*}
	for all $\nu \in \argmin \cG$, where $\tilde{\epsilon}_n := C \epsilon_n$, for all $n\in \N$, with $C$ the constant in \eqref{eq:square_inequality_PG}. Since $\sum_n \tilde{\epsilon}_n = C \cdot \sum_n \epsilon_n <+\infty$ we obtain by summing up
	\begin{equation}
		\sum_n (1-\tau_nL)W_2^2(\cT_{\tau_n}(\mu_n),\mu_n)<+\infty.
	\end{equation}
	Since by hypothesis we have $\sup_i \tau_i < \frac{1}{L}$, it holds in particular $W_2(\cT_{\tau_n}(\mu_n),\mu_n)\to 0$.
\end{proof}

From now on, we will use the notation $\tilde{\epsilon}_n := C \epsilon_n$, for all $n\in \N$, with $C$ the constant given by the previous lemma.

\begin{lemma}\label{lem:second_ingredient_PG}
	Let $\mu \in \cP_2(\R^d)$, $\eta := (I-\tau \nabla F)_{\#}\mu$ and $\bar{\mu}^+:=J_{\tau,\mathcal{H}}(\eta)$, with $\tau < \frac{1}{L}$. Then, for every $\bar \mu \in \cP_2^r(\R^d)$ with $W_2(\bar \mu, \mu)\leq \epsilon$,  we have that
	\begin{equation}\label{eq:lem_second_ingredient_PG}
		\mathcal{G}(\bar \mu^+) - \mathcal{G}(\bar \mu) \leq  \frac{\epsilon}{\tau} (W_2(\bar \mu, \eta)+W_2(\bar \mu^+, \eta)+\epsilon)
	\end{equation}
\end{lemma}
\begin{proof}
	Let $\bar \mu \in \cP_2^r(\R^d)$ such that $W_2(\bar \mu, \mu)\leq \epsilon$. For every $\bar \eta\in \cP_2^r(\R^d)$ such that $W_2(\bar \eta, \eta)\leq \delta$, we have that
	\begin{equation*}
		\begin{aligned}
			\mathcal{H}(\bar{\mu}^+)-\mathcal{H}(\bar{\mu}) & \leq \frac{1}{2\tau}W_2^2(\bar \mu,\eta)-\frac{1}{2\tau}W_2^2(\bar \mu^+,\eta)\\
			& \leq \frac{1}{2\tau}W_2^2(\bar \mu,\bar \eta) - \frac{1}{2\tau} W_2^2(\bar \mu^+, \bar \eta) + C_\delta\frac{\delta}{\tau}\\
			& = \frac{1}{2\tau}\int \|T_{\bar \eta}^{\bar \mu}(z) - z\|^2 - \|T_{\bar \eta}^{\bar \mu^+}(z) - z\|^2 \, d\bar \eta (z) + C_\delta \frac{\delta}{\tau}\\
			& = \frac{1}{2\tau}\int \|T_{\bar \eta}^{\bar \mu}(z)\|^2 - \|T_{\bar \eta}^{\bar \mu^+}(z)\|^2 + 2\langle z, T_{\bar \eta}^{\bar \mu^+}(z)-T_{\bar \eta}^{\bar \mu}(z) \rangle \, d\bar \eta (z) + C_\delta \frac{\delta}{\tau},
		\end{aligned}
	\end{equation*}
	where $C_\delta$ can be choosen as $C_\delta=W_2(\bar \mu, \eta)+W_2(\bar \mu^+,\eta)+\delta$. On the other hand
	\begin{equation*}
		\begin{aligned}
			\mathcal{E}_F(\bar{\mu}^+)-\mathcal{E}_F(\bar{\mu})& = \int  F(T_{\bar \eta}^{\bar \mu^+}(z))-F(T_{\bar \eta}^{\bar \mu}(z)) \, d \bar \eta (z) \\
			& \hspace{-1cm}\leq \int \langle \nabla F(T_{\bar \eta}^{\bar \mu}(z)), T_{\bar \eta}^{\bar \mu^+}(z) - T_{\bar \eta}^{\bar \mu}(z) \rangle + \frac{L}{2} \|T_{\bar \eta}^{\bar \mu^+}(z)-T_{\bar \eta}^{\bar \mu}(z)\|^2 \, d \bar \eta (z)\\
			& \hspace{-1cm} = \frac{1}{2\tau} \int \bigg( - 2 \langle (I-\tau \nabla F)(T_{\bar \eta}^{\bar \mu}(z)), T_{\bar \eta}^{\bar \mu^+}(z) - T_{\bar \eta}^{\bar \mu}(z) \rangle \\ 
			& \hspace{1cm} - 2 \langle T_{\bar \eta}^{\bar \mu}(z), T_{\bar \eta}^{\bar \mu}(z) - T_{\bar \eta}^{\bar \mu^+}(z) \rangle + \tau L \|T_{\bar \eta}^{\bar \mu^+}(z)-T_{\bar \eta}^{\bar \mu}(z)\|^2 \bigg)\, d \bar \eta (z).
		\end{aligned}
	\end{equation*}
	Putting all together leads to
	\begin{equation}\label{eq:Gbarmu+-Gbarmu}
		\begin{aligned}
			\mathcal{G}(\bar \mu^+) - \mathcal{G}(\bar \mu) & \leq \frac{1}{2\tau} \int \bigg(2 \langle z -(I-\tau \nabla F)(T_{\bar \eta}^{\bar \mu}(z)),  T_{\bar \eta}^{\bar \mu^+}(z)-T_{\bar \eta}^{\bar \mu}(z) \rangle \\
			& \hspace{3.7cm} - (1-\tau L)\|T_{\bar \eta}^{\bar \mu^+}(z)-T_{\bar \eta}^{\bar \mu}(z)\|^2\bigg) \, d \bar \eta (z) + C_\delta\frac{\delta}{\tau}
		\end{aligned}
	\end{equation}
	Since $\bar \eta$ was arbitrary, we can actually choose $\bar \eta = (I-\tau \nabla F)_{\#} \bar \mu$. With this choice, using the fact that the map $I-\tau \nabla F$ is nonexpansive, we have (see for example \cite[Proposition 4.2]{Berdellima2025})
	\[W_2(\bar \eta , \eta)=W_2((I-\tau \nabla F)_{\#} \bar \mu , (I-\tau \nabla F)_{\#} \mu)\leq W_2(\bar \mu , \mu)\leq \epsilon,\]
	and we can set $\delta = \epsilon$. Defining $\phi = \frac{1}{2}\|\cdot\|^2-\tau F$, we have $\nabla \phi = I-\tau \nabla F$ and since $\nabla F$ is $L$-Lipschitz and $\tau < \frac{1}{L}$, then $\phi = \frac{1}{2}\|\cdot\|^2-\tau F$ is strongly convex. This implies that $(I-\tau \nabla F)^{-1}$ is the gradient of a convex function and it is Lipschitz continuous (and thus also in $L^2(\bar \eta)$). We can therefore apply \Cref{thm:brenier} (ii) and obtain $T_{\bar \eta}^{\bar \mu}=(I-\tau \nabla F)^{-1}$ and $ z -(I-\tau \nabla F)(T_{\bar \eta}^{\bar \mu}(z))=0$ for $\bar \eta$-almost every $z$. 
	
	Finally, \eqref{eq:Gbarmu+-Gbarmu} yields
	\begin{equation}
		\mathcal{G}(\bar \mu^+) - \mathcal{G}(\bar \mu) \leq 0-\frac{1-\tau L}{2\tau} \int \|T_{\bar \eta}^{\bar \mu^+}(z)-T_{\bar \eta}^{\bar \mu}(z)\|^2 \, d \bar \eta (z)+ C_\epsilon \frac{\epsilon}{\tau} \leq C_\epsilon \frac{\epsilon}{\tau}.
	\end{equation}
\end{proof}

\begin{theorem}\label{thm:PG_convergence1}
	Let $\mathcal{H} : \cP_2(\R^d) \to \mathbb{R} \ $ and $F$ satisfying (A1)-(A3) and suppose $\argmin \cG \neq \emptyset$. Let $\{\epsilon_n\}_n\subset \mathbb{R}_{\geq 0}$ with $\sum_{n=0}^{\infty} \epsilon_n < \infty$, $\{\tau_n\}_n \subset \left(0,\frac{1}{L}\right)$ with $\sum_{i=0}^{\infty}\tau_i= \infty$, $\sup_i \tau_i < \frac{1}{L}$ and let $\sigma_n:= \sum_{i=0}^{n-1}\tau_i$, for $n\in\mathbb{N}$. Let $\{\mu_n\}_n$ satisfying
	\[W_2(\mu_{n+1},\cT_{\tau_n}(\mu_n))\leq \epsilon_n, \quad \text{for all } n \in \N.\]
	\begin{enumerate}
		\item Then, we have
		\begin{equation}\label{eq:rate_PG_bar1}
			\cG(\bar \beta_n)-\inf \cG = O\left(\frac{1}{\sigma_n}\right), \quad \text{as } n\to \infty,
		\end{equation}
		where $\bar \beta_n := \cT_{\tau_{j_n}}(\mu_{j_n})$ with $j_n = \argmin_{i=0,\dots,n-1}\{\cG(\cT_{\tau_i}(\mu_i))\}$, defines the sequence of the best iterates.
		\item If $\sum_{n=0}^{\infty}\frac{\sigma_n}{\tau_n}\epsilon_{n-1}  < \infty$, then \[\cG(\cT_{\tau_n}(\mu_n))-\inf \cG = O\left(\frac{1}{\sigma_n}\right), \quad \text{for } n\to \infty,\] and $\{\mu_n\}_n$ converges with respect to the topology $\tau_{w,2}$ to some $\mu^*\in \argmin \cG$. In particular we have $W_p(\mu_n,\mu^*)\to 0$ for all $p\in[1,2)$.
	\end{enumerate}
\end{theorem}
\begin{proof}
	The convergence analysis follows similar steps as the ones depicted in \Cref{sec:first_error}. Summing the EVI \eqref{eq:discrete_EVI_FB_true} and using \eqref{eq:square_inequality_PG}, we have the first main ingredient, corresponding to \eqref{eq:for_rates1}
	\begin{equation}\label{eq:for_rates1_PG}
		\begin{aligned}
			2\sum_{n=0}^{N-1}\tau_n\cG(\cT_{\tau_n}(\mu_n))+W_2^2(\mu_N,\nu)\leq 2&\sigma_N \cG(\nu)+W_2^2(\mu^0,\nu)\\
			&-\sum_{n=0}^{N-1}(1-\tau_n L)W_2^2(\cT_{\tau_n}(\mu_n),\mu_n)+\sum_{n=0}^{N-1}\tilde{\epsilon}_n.
		\end{aligned}
	\end{equation}
	From \Cref{lem:second_ingredient_PG} we obtain the second main ingredient, corresponding to \eqref{eq:decrease_JKO}. In fact, for all $n\in \mathbb{N}$, letting $\mu=\mu_n$, $\bar \mu^+= \cT_{\tau_{n}}(\mu_{n})$ and $\bar \mu= \cT_{\tau_{n-1}}(\mu_{n-1})$ we obtain from \eqref{eq:lem_second_ingredient_PG} and the fact that the sequences in play are bounded, that there exists $C>0$ such that
	\begin{equation}
		\mathcal{G}(\cT_{\tau_{n}}(\mu_{n})) \leq \mathcal{G}(\cT_{\tau_{n-1}}(\mu_{n-1})) +  C\frac{\epsilon_{n-1}}{2\tau_n}, \quad \text{ for all } n \in \N.
	\end{equation}
	With these ingredients, it is possible to conclude similarly to \Cref{thm:JKO_convergence1}.
\end{proof}

\begin{remark}
		In the case $\lambda > 0$, using \eqref{eq:discrete_EVI_FB_true} and \eqref{eq:square_inequality_PG}, we obtain for all $\nu \in \argmin \cG$ that
		\[W_2^2(\mu_{n+1}, \nu)\leq (1-\tau_n \lambda)W_2^2(\mu_n,\nu)+\tilde{\epsilon}_n.\]
		From this, strong convergence results in $W_2$ can be achieved. However, since at each iteration the error $\epsilon_n$ is committed, we cannot expect linear convergence rates. This motivates us not to treat the case $\lambda > 0$ separately from the case $\lambda = 0$. A similar reasoning applies when assuming strong convexity along generalized geodesics of the functional $\mathcal{H}$, see also Remark~\ref{rem:strongly-convex}.
	\end{remark}
	\begin{remark}[Ergodic convergence]
		Similar considerations to those in Remark~\ref{rmk:ergodic} apply in this setting as well. In particular, under similar conditions to those in Remark~\ref{rmk:ergodic}, the convergence rate in \eqref{eq:rate_PG_bar1} can also be obtained for the Wasserstein barycenter sequence $\{\bar \beta_n\}_n$ formed from the first $n$ elements of the sequence $\{S_{\tau_i}(\mu_i)\}_i$, with weights $\{\frac{\tau_i}{\sigma_n}\}_0^{n-1}$. This result can hold without the additional assumptions required in point~2 of \Cref{thm:PG_convergence1}.
\end{remark}

With \Cref{thm:PG_convergence1} we can generalize the result by Diao, Balasubramanian, Chewi and Salim \cite[Theorem 5.3]{diao2023}. In their work they prove weak convergence for the proximal-gradient algorithm but they only work with Gaussians and the so-called Bures-Wasserstein space.
\begin{corollary}\label{cor:fb_convergence}
	Let $\mathcal{H} : \cP_2(\R^d) \to \mathbb{R} \ $ and $F$ satisfying (A1)-(A3) and suppose $\argmin \cG \neq \emptyset$. Let $\mu_{0}\in \cP_2(X)$ and $\{\mu_n\}_n$ satisfying
	\[\mu_{n+1}=J_{\tau , \calH}\circ (I-\tau \nabla F)_{\#}(\mu_n).\]
	Then $\{\mu_n\}_n$ converges with respect to the topology $\cT_{w,2}$ (and thus also narrowly) to some $\mu^*\in \argmin \cG$.
\end{corollary}
\begin{proof}
	We can apply the previous result with $\epsilon_n=0$ and $\tau_n=\tau>0$, for all $n\in \N$.
\end{proof}

\begin{theorem}\label{thm:PG_convergence2}
	Let $\mathcal{H} : \cP_2(\R^d) \to \mathbb{R} \ $ and $F$ satisfying (A1)-(A3) and suppose $\argmin \cG \neq \emptyset$. Let $\{\epsilon_n\}_n\subset \mathbb{R}_{\geq 0}$ with $\sum_{n=0}^{\infty} \epsilon_n < \infty$, $\{\tau_n\}_n \subset \left(0,\frac{1}{L}\right)$ with $\sum_{i=0}^{\infty}\tau_i= \infty$, $\sup_i \tau_i < \frac{1}{L}$ and let $\sigma_n:= \sum_{i=0}^{n-1}\tau_i$, for $n\in\mathbb{N}$. Let $\{\eta_n\}_n$ and $\{\mu_n\}_n$ satisfying $\eta_{n}=(I-\tau \nabla F)_{\#}(\mu_n)$ and
	\begin{equation}\label{eq:second_error_final}
		\mathcal{H}(\mu_{n+1})+\frac{1}{2\tau}W_2^2(\mu_{n+1},\eta_{n})\leq \mathcal{H}(J_{\tau_n,\mathcal{H}}(\eta_{n}))+\frac{1}{2\tau}W_2^2(J_{\tau_n,\mathcal{H}}(\eta_{n}),\eta_{n})+ \frac{\epsilon_n^2}{2\tau_n},
	\end{equation}
	\begin{enumerate}
		\item It holds the rate
		\begin{equation}\label{eq:rate_PG_bar2}
			\cG(\beta_n)-\inf \cG = O\left(\frac{1}{\sigma_n}\right), \quad \text{as } n\to \infty,
		\end{equation}
		where $\beta_n := \mu_{j_n}$ with $j_n = \argmin_{i=0,\dots,n-1}\{\cG(\mu_i)\}$, defines the sequence of the best iterate.
		\item If $\sum_{n=0}^{\infty}\frac{\sigma_n}{\tau_n}\epsilon_{n}  < \infty$, then \[\cG(\mu_n)-\inf \cG = O\left(\frac{1}{\sigma_n}\right), \quad \text{for } n\to \infty,\] and $\{\mu_n\}_n$ converges with respect to the topology $\tau_{w,2}$ to some $\mu^*\in \argmin \cG$. In particular we have $W_p(\mu_k,\mu^*)\to 0$ for all $p\in[1,2)$.
	\end{enumerate}
\end{theorem}
\begin{proof}
	Following the first part of \Cref{thm:second_error_choice}, we obtain similarly from the condition \eqref{eq:second_error_final}, that
	\[W_2(\cT_{\tau_n}(\mu_n),\mu_{n+1})=W_2(J_{\tau_n,\mathcal{H}}(\eta_n),\mu_{n+1})\leq \epsilon_n.\]
	With this, we can already prove all the results of \Cref{thm:PG_convergence1}. On the other hand, the inequality \eqref{eq:second_error_final} implies \begin{equation*}
		\begin{aligned}
			2\tau_n\mathcal{H}(\cT_{\tau_n}(\mu_{n})) & \geq 2\tau_n\mathcal{H}(\mu_{n+1})+W_2^2(\mu_{n+1},\eta_{n})- W_2^2(\cT_{\tau_n}(\mu_{n}),\eta_{n})- \epsilon_k^2\\
			& \hspace{-1cm}\geq 2\tau_n\mathcal{H}(\mu_{n+1})-\left(W_2(\mu_{n+1},\eta_{n})+ W_2(J_{\tau}(\eta_{n}),\eta_{n})\right)W_2(\mu_{n+1},\cT_{\tau_n}(\mu_n)) - \epsilon_n^2 \\
			& \hspace{-1cm}\geq 2\tau_n\mathcal{H}(\mu_{n+1})-c_1\epsilon_n - \epsilon_n^2,
		\end{aligned}
	\end{equation*}
	for some $c_1>0$. Setting $\bar \mu_{n+1}:=\cT_{\tau_n}(\mu_{n})$ we also have 
	\begin{equation*}
		\begin{aligned}
			\mathcal{E}_F(\bar \mu_{n+1}) & = \int F(x) \, d\bar \mu_{n+1}(x) \\ 
			&\geq \int F\left(T_{\bar \mu_{n+1}}^{\mu_{n+1}}(x)\right) +\langle \nabla F\left(T_{\bar \mu_{n+1}}^{\mu_{n+1}}(x)\right),x-T_{\bar \mu_{n+1}}^{\mu_{n+1}}(x)  \rangle \, d\bar \mu_{n+1}(x)\\
			&	\geq \mathcal{E}_F(\mu_{n+1}) - \left(\int \left\|\nabla F(x')\right\|^2\, d \mu_{n+1}(x')\right)^{\frac{1}{2}} W_2(\mu_{n+1},\cT_{\tau_n}(\mu_n)).
		\end{aligned}
	\end{equation*}
	By noticing that $\|\nabla F(x')\|^2 \leq 2\|\nabla F(x')-\nabla F(0)\|^2+2\|\nabla F(0)\|^2 \leq L^2 \|x'\|^2 + 2\|\nabla F(0)\|^2$, and $\{\mu_n\}_n\subset \cP_2(\R^d)$ is bounded, and $\tau_n < \frac{1}{L}$, there exists a constant $c_2>0$ such that
	\begin{equation*}
		\begin{aligned}
			2\tau_n\mathcal{E}_F(\bar \mu_{n+1})	\geq 2\tau_n\mathcal{E}_F(\mu_{n+1}) - c_2 \epsilon_n
		\end{aligned}
	\end{equation*}
	and thus
	\begin{equation}\label{eq:ineq_second_error_PG}
		\begin{aligned}
			2\tau_n\mathcal{G}(\cT_{\tau_{n}}( \mu_{n})) \geq 2\tau_n\mathcal{G}(\mu_{n+1}) - c_1\epsilon_n - \epsilon_n^2 - c_2 \epsilon_n
		\end{aligned}
	\end{equation}
	Combining this with \eqref{eq:for_rates1_PG}, we obtain
	\begin{equation}\label{eq:almost_final_second_err_PG}
		2\sum_{n=0}^{N-1}\tau_n\cG(\mu_{n+1})\leq 2\sigma_N \cG(\nu)+W_2^2(\mu^0,\nu)+\sum_{n=0}^{N-1}\tilde{\epsilon}_n+(c_1+c_2)\sum_{n=0}^{N-1}\epsilon_n + \sum_{n=0}^{N-1}\epsilon_n^2.
	\end{equation}
	Since by \eqref{eq:second_error_final} we also have $\mu_n\in \cP_2^r(\R^d)$ for all $n\in \N$, we can use \Cref{lem:second_ingredient_PG} with $\bar \mu^+=\cT_{\tau_{n}}( \mu_{n})$, $\bar \mu = \mu_n$ and obtain
	\begin{equation}
		\mathcal{G}(\cT_{\tau_{n}}( \mu_{n})) \leq \mathcal{G}(\mu_n) + c_3\frac{\epsilon_n}{2\tau_n}
	\end{equation}
	for some $c_3 > 0$. Combining this with \eqref{eq:ineq_second_error_PG}, we arrive to
	\begin{equation*}
		\mathcal{G}(\mu_{n+1}) \leq \mathcal{G}(\mu_n) + (c_1+c_2+c_3)\frac{\epsilon_n}{2\tau_n} + \frac{\epsilon_n^2}{2\tau_n}.
	\end{equation*}
	We denote by $C = c_1+c_2+c_3+\max_i \epsilon_i$, multiply by $\sigma_n$, and obtain
	\begin{equation}
		\begin{aligned}
			(\sigma_{n+1}-\tau_n)\cG(\mu_{n+1}) - \sigma_{n}\cG(\mu_n) \leq C\frac{\sigma_n}{2\tau_n}\epsilon_n.
		\end{aligned}
	\end{equation}
	Summing up from $0$ to $N-1$ and recalling that $\sigma_0=0$, we have
	\begin{equation}
		\begin{aligned}
			\sigma_{N}\cG(\mu_N) - C \sum_{n=0}^{N-1}\frac{\sigma_n}{2\tau_n}\epsilon_{n} \leq \sum_{n=0}^{N-1}\tau_n \cG(\mu_{n+1}).
		\end{aligned}
	\end{equation}
	which combined with \eqref{eq:almost_final_second_err_PG} leads to
	\begin{equation*}
		\cG(\mu_{N}) - \cG(\nu) \leq \frac{1}{2\sigma_N}\left(W_2^2(\mu^0,\nu)+\sum_{n=0}^{N-1}\left(\tilde{\epsilon}_n+(c_1+c_2)\epsilon_n + \epsilon_n^2+C\frac{\sigma_n}{\tau_n}\epsilon_{n}\right)\right),
	\end{equation*}
	which concludes the proof.
\end{proof}

\section*{Conclusions}

In this paper, we studied the convergence properties of inexact Jordan-Kinderlehrer-Otto (JKO) schemes and proximal-gradient algorithms in Wasserstein spaces. We focused on settings where obtaining exact solutions to iterative minimization problems is impractical and introduced controlled approximations for both the Wasserstein distance and the energy functionals.
We provided rigorous convergence guarantees, demonstrating that weak convergence remains attainable in the presence of inexact computations. Additionally, we extended our analysis to proximal-gradient algorithms. Our findings lay the groundwork for broader applicability of these schemes in practical settings.  We also incorporated the flexibility of varying stepsizes, leading to new convergence insights. This study also raises several compelling questions. Notably, it highlights the importance of quantifying the approximation behavior of existing methods, such as ALG2 in \cite{BCL16}, when applied to solve \eqref{eq:prox_point_intro}. Future research may pursue this direction by analyzing a range of existing algorithms and proposing new ones for solving \eqref{eq:prox_point_intro} or the associated saddle point problem derived from the Benamou–Brenier formulation, with quantitative approximation results. This analysis can be complemented by estimates concerning approximations coming from regularized formulations as in \cite{CDPS17,LLO23, LLW20, LWL24}. Further error estimates may also be developed to account for discretization of measures and, in the Benamou–Brenier case, temporal discretization.

The regularity property (A3) seems essential for the analysis carried out in \cite{WPG2020} as well as for our current analysis,  but we conjecture that it might be weakened or removed altogether by employing the subdifferential in \cite[Theorem 10.3.6]{AGS08}.  Investigating this possibility is a promising avenue for future work.

\section*{Acknowledgments}

S.D.M, E.N. and S.V. are supported by the MUR Excellence Department Project awarded to the department of mathematics at the University of Genoa, CUP D33C23001110001, by the US Air Force Office of Scientific Research (FA8655-22-1-7034). S.V. acknowledges the support of the European Commission (grant TraDE-OPT 861137). S.D.M. and S.V. acknowledge the support of the Ministry of Education, University and Research (grant BAC FAIR PE00000013 funded by the EU - NGEU) and the financial support of the PRIN 202244A7YL. S.D.M., E.N. and S. V. are members of the Gruppo Nazionale per l’Analisi Matematica, la Probabilità e le loro Applicazioni (GNAMPA) of the Istituto Nazionale di Alta Matematica (INdAM). This work represents only the view of the authors. The European Commission and the other organizations are not responsible for any use that may be made of the information it contains.

\bibliography{biblio} 
\bibliographystyle{plain}

\end{document}